\documentclass{amsart}

\usepackage{float}
\usepackage{amsthm}
\usepackage{amsmath}
\usepackage{amsfonts}
\usepackage{amssymb}
\usepackage{latexsym}
\usepackage{mathrsfs}
\usepackage{geometry}
\usepackage{graphicx}
\usepackage{color, soul}

\newcommand{\sca}[1]{\left\langle#1\right\rangle} 
\newcommand{\abs}[1]{\lvert#1\rvert}

\newtheorem{theorem}{Theorem}[section]
\newtheorem{lemma}[theorem]{Lemma}
\newtheorem{proposition}[theorem]{Proposition}
\newtheorem{corollary}[theorem]{Corollary}
\newtheorem*{collatz}{The $3x+1$-problem}

\theoremstyle{definition}

\theoremstyle{remark}
\newtheorem{example}[theorem]{Example}
\newtheorem{remark}[theorem]{Remark}

\numberwithin{equation}{section}

\begin{document}

\title{Koopman operators and the $3x+1$-dynamical system.}

\author{John Leventides}
\address{Department of Economics, Faculty of Economics and Political Sciences, National and Kapodistrian University of Athens.}
\email{ylevent@econ.uoa.gr}

\author{Costas Poulios}
\address{Department of Economics, Faculty of Economics and Political Sciences, National and Kapodistrian University of Athens.}
\email{konpou@econ.uoa.gr}


\date{}

\begin{abstract}
The $3x+1$-problem (or Collatz problem) is a notorious conjecture in arithmetic. It can be viewed as iterating a map and, therefore, it is a dynamical system on the discrete space $\mathbb{N}$ of natural numbers. The emerging dynamical system is studied in the present work with methods from the theory of Koopman operators and $C^*$-algebras. This approach enables us to ``lift" the $3x+1$-dynamical system from the state space (i.e the set $\mathbb{N}$) to spaces of functions defined on the state space, i.e. to sequence spaces. The advantage of this lifting is that the Collatz problem can be described via bounded linear operators, which consist an extensively studied area of Analysis. We study the properties of these operators and their relationship to the $3x+1$-problem. Furthermore, we use Fourier transform techniques to investigate the frequency content of the sequences of signs emerging from the trajectories of the Collatz map. This enables us to define an isometry on a Hilbert space. Finally, we utilize the $C^*$-algebra generated by this isometry in order to study how the sequences of signs correlate with each other.
\end{abstract}

\maketitle

\section{Introduction}\label{sec.introduction}
The Collatz conjecture or $3x+1$-problem is one of those mathematical problems which are very simple to be formulated but extremely difficult to be solved. In order to state this problem, we only need the Collatz function which is defined on integers with the following simple way: if an integer is odd ``multiply by three and add one", while if it is even ``divide by two". In other words, one has
$$C(n)= \left\{
          \begin{array}{ll}
            \frac{n}{2}, & \hbox{$n\equiv 0\mod 2$;} \\
            3n+1, & \hbox{$n\equiv 1\mod 2$.}
          \end{array}
        \right.
$$
Given any positive integer $n$, we are interested in the \emph{trajectory} or \emph{orbit} of $n$ under the function $C$, that is the sequence $(n, C(n), C^2(n), C^3(n), \ldots, C^k(n), \ldots)$ produced by $C$ under iteration. The Collatz conjecture asserts that, starting from any $n$, the above-mentioned trajectory eventually reaches the number $1$, which (easily) implies that thereafter iterations cycle taking successive values $1,4,2,1,\ldots$.

It is a very common technique in the literature to replace the Collatz function with a slight modification of it, namely the function
$$T(n)= \left\{
          \begin{array}{ll}
            \frac{n}{2}, & \hbox{$n\equiv 0\mod 2$;} \\
            \frac{3n+1}{2}, & \hbox{$n\equiv 1\mod 2$.}
          \end{array}
        \right.
$$
The map $T$ is more convenient than $C$ in a number of ways and the majority of the known results is expressed in terms of the function $T$. Each iterate of $T$ performs one division by $2$ and it is not hard to see that the trajectory $\{n, T(n), \ldots, T^k(n), \ldots \}$ simply omits some steps of the corresponding trajectory generating by $n$ under iteration of $C$. In the present paper, the term Collatz function (or Collatz map) will always refer to the map $T$. In this framework, the Collatz conjecture is equivalent to the following statement.

\begin{collatz}
For every positive integer $n$, there is $k_0=k_0(n) \in \mathbb{N}$ (depending on $n$) such that $T^{k_0}(n)=1$ (which means that after this point iterations of $T$ cycle taking the values $1,2,1,\ldots$).
\end{collatz}

Let us now discuss the possible behaviour of a trajectory coming from iteration of the map $T$. Fix any positive integer $n$. Then, two cases occur: either
(a)~there are $k_0, m$ such that $T^{k_0+m}(n) = T^{k_0}(n)$, or
(b)~$T^k(n) \ne T^m(n)$ for any $k\neq m$.
In the second case, it is clear that the sequence $(T^k(n))_{k\in \mathbb{N}}$ diverges to infinity. If the first case occurs, then the sequence
$(T^k(n))_{k\in \mathbb{N}}$ end in the cycle $$T^{k_0}(n), T^{k_0+1}(n), T^{k_0+2}(n), \ldots, T^{k_0+m-1}(n),$$
which repeats itself. It is not hard to see that when the ``length" $m$ of this cycle is $2$ then either $T^{k_0}(n)$ or $T^{k_0+1}(n)$ is equal to $1$.
Consequently, for the behaviour of the Collatz function there are three possibilities.
\begin{enumerate}
    \item The Collatz conjecture holds true; or,
    \item there is (at least) one positive integer $n$ such that the sequence $(T^k(n))_{k\in \mathbb{N}}$ ends in some attractor of length bigger than $2$; or,
    \item there is (at least) one positive integer $n$ such that $(T^k(n))_{k\in\mathbb{N}}$ diverges to infinity.
\end{enumerate}
Therefore, if one wishes to solve the Collatz conjecture, then they should exclude the last two cases. However, this remark does not facilitate the solution of the problem. Despite the remarkable work of many scientists, the $3x+1$ problem has been proved extremely resilient to any attacks so far. Several areas of mathematics have been utilized in the effort the $3x+1$ problem to be solved. Indicatively, we can mention number theory, dynamical systems, ergodic theory, stochastic models, probability theory etc. Furthermore, computer programs have also been employed and have shown that the conjecture is true for billions of numbers. We refer to \cite{Lag-1} and \cite{Lagbook} for a survey on the $3x+1$-problem, and to \cite{Lag-2} for a bibliography on the subject.

The purpose of this work is to apply ideas and techniques from the Theory of Operators in order to the study the dynamical system emerging from the $3x+1$-function, i.e.
$$x_{n+1}=T(x_n), \quad n\in\mathbb{N}^*.$$
A standard approach towards this direction is by exploiting the so-called composition or Koopman operator. This type of operators represents an increasingly popular formalism of dynamical systems and it enables the analysis and prediction especially of nonlinear systems. The core idea underlying the Koopman operator theory is as follows. Assume that we are given a dynamical system
$$x^+ = S(x), \quad x\in M,$$
defined on the state space $M$. (In the case of the $3x+1$-dynamical system, we have $S=T$ and $M=\mathbb{N}$.) Any function $g\colon M\to \mathbb{R}$ is called an \emph{observable} of the system. The set of all observables forms a vector space. The Koopman operator or composition operator, denoted by $\mathcal{K}$, is a linear operator defined on this vector space and it is given by
$$ \mathcal{K}(g) = g\circ S,$$
where $\circ$ is the composition of functions. Roughly speaking, the operator $\mathcal{K}$ updates every observable $g$ according to the evolution of the initial dynamical system. The main advantage of the Koopman operator is its linearity which follows easily form its definition as a composition operator. However, it should pointed out that the space of all observables is infinite dimensional. In other words, a ``trade off'' takes place, where we exchange linearity with dimensions. (On account of the fact that the $3x+1$-dynamical system is of discrete-time nature, we focus on dynamical system of this
type. However, Koopman operators are also applied to continuous-time systems and the analysis is quite similar. We refer to \cite{arba}, \cite{budi} and \cite{mez} for more information concerning Koopman operator theory and its applications.)

The rest of the paper is organised as follows. In Section \ref{sec.Collatz} we summarize some basic fact concerning the Collatz map $T$. In Section \ref{sec.density_theorems}, we review several density theorems related to the $3x+1$ problem. In Sections \ref{sec-Koopman_operators}, \ref{sec. spectral radius} and \ref{sec.ell-1-operator}, the forward and backward Collatz-Koopman operators on $\ell_p$, $1\le p\le\infty$ are introduced and their properties and connection to the $3x+1$-problem are studied. In Section \ref{sec.periodicity}, we analyze the double indexed sequence of signs $((-1)^{T^k(n)} )_{k,n=1}^\infty$ associated with the orbits of the Collatz map. Based on this analysis, we define in Section \ref{sec.ell-2-oper} an isometry on a Hilbert space and in Section \ref{sec.Wold-decomp} we describe its Wold-von Neumann decomposition. Finally, in Section \ref{sec.C^star-alg} we utilize the $C^*$-algebra generated by the aforementioned isometry in order to study how the sequences of signs correlate with each other.

The present work has a two-pronged aim. Firstly, we believe that the results of the paper contribute to the better understanding of the $3x+1$-problem and, secondly, that our work may be applied to other dynamical systems and, thus, it will provide new tools for the study of such systems.

\section{Basic properties of the $3x+1$-dynamical system}\label{sec.Collatz}
In this section, we review and, in some cases, refine a few basic facts about the $3x+1$ dynamical system. First of all, we consider the next double indexed sequence. For any $i=1,2,3,\ldots $ and any $n=1,2,\ldots$, we set
$$x_{ni} = T^{i-1}(n) \mod 2,$$
where $T^0(n)=n$ for every positive integer $n$, $T^1=T$ and $T^m = T\circ T \circ \ldots \circ T$. The sequence $(x_{nk})_{n,k=1}^\infty$ determines completely the trajectories of the $3x+1$-map $T$. This statement becomes clear and more accurate with the following result which can be proved by induction on $k$.

\begin{theorem}\label{th.formula_for_T^k}
For any $k=0,1,2\ldots$ and any $n=1,2,\ldots$ the following formula holds:
\begin{equation}\label{eq.colla-1}
T^k(n) =\frac{3^{\sum_{i=1}^kx_{ni}}}{2^k} \cdot n + \frac{1}{2^k} \sum_{i=1}^{k} x_{ni} 2^{i-1} 3^{\sum_{\lambda=i+1}^{k}x_{n\lambda}}.
\end{equation}
\end{theorem}

The importance of the sequence $(x_{nk})_{n,k=1}^\infty$ for the
$3x+1$-problem became clear very early in the investigation of the
problem and it is usually called the ``parity sequence'' (e.g. see
\cite{ever}). Although it is hard to find some structure (if any
structure exists) in the sequence $(x_{nk})_{k=1}^\infty$ for an
arbitrary but fixed $n$ (and consequently to the trajectories
$\{T^k(n)\}_{k=1}^\infty$ of the Collatz map), the structure of the
sequence $(x_{nk})_{n=1}^\infty$ for any fixed $k$ is very easy and
clear. This sequence is periodic with period equal to $2^k$ (see
\cite{ever}). However, a little more can be said about the structure
of $(x_{nk})_{n=1}^\infty$ and we will make use of the following
proposition.

\begin{proposition}\label{prop.structure-x_kn}
\begin{enumerate}
  \item For any $k,n=1,2,\ldots$, we have $ x_{n+2^{k-1},k} = 1+ x_{nk}.$

  \item   For any $k=1,2,\ldots$, the sequence $(x_{nk})_{n=1}^\infty$ is periodic with period $2^k$. Furthermore, if we set $\mathbf{y_k} = (x_{nk})_{n=1}^{2^{k-1}}$ then the sequence $(x_{nk})_{n=1}^\infty$ is built periodically by the vector $(\mathbf{y_k}, \mathbf{1}+\mathbf{y_k}).$
\end{enumerate}
\end{proposition}

\begin{proof}
The second assertion is an immediate consequence of the first one. The first assertion can be proved easily by induction on $k$. Indeed, for $k=1$, the result is clear, since $x_{n1} = n \mod2$ and $x_{n+1,1}=(n+1)\mod2=1+x_{n1}.$

Assume now that for any $i=1,2,\ldots, k$ and any $n=1,2,\ldots$, we have $x_{n+2^{i-1},i}=1+x_{ni}$. We show that $x_{n+2^k,k+1}=1+x_{n,k+1}$ for every $n$. It suffice to prove that $T^{k}(n+2^{k}) - T^{k}(n) = 1 \mod 2$. The inductive hypothesis implies that $x_{n+2^{i},i}=x_{ni}$ for any $i=1,2,\ldots,k$ and $n=1,2,\ldots$. Since $2^k$ is an integer multiple of $2^i$ for any $i=1,2,\ldots,k$, we obtain:
$$x_{n+2^{k},i}=x_{ni} \quad \text{for all } i=1,2,\ldots, k \text{ and } n=1,2,\ldots.$$
The above equation and Theorem \ref{th.formula_for_T^k} imply, after some routine calculations, that
$$
  T^{k}(n+2^{k}) - T^{k}(n) = 3^{\sum_{i=1}^{k} x_{ni}}= 1\mod 2,
$$
and the result follows.
\end{proof}
By Theorem \ref{th.formula_for_T^k} and Proposition
\ref{prop.structure-x_kn}, we deduce immediately the next corollary.

\begin{corollary}\label{cor.Collatz-function}
 Let $k,n$ be any positive integers. We write $n=2^k p+ \upsilon$, where $p,\upsilon \in \mathbb{N}$ with $0\le \upsilon<2^k$.
Then, the following equation hods:
$$T^k(n) = \frac{3^{d_{\upsilon k}} \cdot n + \phi_{\upsilon k}}{2^k},$$
where $d_{\upsilon k}$ and $\phi_{\upsilon k}$ depend only on $k$ and the
remainder $\upsilon$ and, more precisely, they are given by:
$$d_{\upsilon k} = \sum_{i=1}^k x_{\upsilon i} \quad \text{and} \quad \phi_{\upsilon k}=\sum_{i=1}^{k} x_{\upsilon i} \cdot 2^{i-1} \cdot 3^{\sum_{\lambda=i+1}^k x_{\upsilon\lambda}}.$$
Furthermore, the above equation can be written as $ T^k(n) = 3^{d_{\upsilon k}}  p +T^k(\upsilon)$. In the case where $\upsilon=0$, that is $n=2^k \cdot p$ is a multiple of $2^k$, then $T^k(n) = p$.
\end{corollary}

Next we consider the infinite dimensional matrix $ [ x_{ni} ]_{n=1, i=1}^{\infty, k}$. The $i$-th column of this matrix has period $2^i$. Hence, the matrix itself has period $2^k$ in the sense that the first $2^k$ rows build periodically the matrix. In other words the matrix is build periodically by the $2^k\times k$-matrix $B_k = [ x_i^n ]_{n=1, i=1}^{2^k, k}$. These matrices in turn are build by the vectors $(\mathbf{y_i})_{i=1}^\infty$ as follows:
$$B_1= \left[\begin{array}{c}
               1 \\
               0
             \end{array} \right]
=\left[\begin{array}{c}
               \mathbf{y_1} \\
               1+\mathbf{y_1}
             \end{array} \right]$$
and
$$B_{k+1}= \left[\begin{array}{cc}
                   B_k & \mathbf{y_{k+1}}^t \\
                   B_k & \mathbf{1}^t+ \mathbf{y_{k+1}}^t
                 \end{array} \right].$$
The next result (see \cite{ever}) is that the rows of the matrix $B_k$ are exactly all the elements of $\mathbb{Z}_2^k$. We consider the $2^k\times k$ matrix $A_k$ whose rows are the elements of $\mathbb{Z}_2^k$ in lexicographical order. Then, we have the following.

\begin{theorem}\label{th.rows_of parity matrix}
  For any $k=1,2,\ldots $, the rows of the matrix $B_k$ are a permutation of the rows of the matrix $A_k$.
\end{theorem}

\begin{proof}
 By induction on $k$. For $k=1$ the result is clear since
$$B_1=\left[\begin{array}{c}
               1 \\
               0
             \end{array} \right].$$
Assume now that $B_k$ is a permutation of $A_k$ for some $k=1,2,\ldots$. Then, by construction, the matrix $B_{k+1}$ is given by:
$$B_{k+1}= \left[\begin{array}{cc}
                   B_k & \mathbf{y_{k+1}}^t \\
                   B_k & \mathbf{1}^t+ \mathbf{y_{k+1}}^t
                 \end{array} \right].$$
We observe that $\mathbf{y_{k+1}}^t$ and $\mathbf{1}^t+ \mathbf{y_{k+1}}^t$, which are opposites, are appended in the matrix $B_k$. Hence, every row of $B_k$ enters in $B_{k+1}$ twice, once with the coordinate $0$ and once with the coordinate $1$ attached in the row. By the inductive hypothesis $B_k$ consists of the elements of $\mathbb{Z}_2^k$. Therefore, the procedure for building $B_{k+1}$ amounts into forming the cartesian product $\mathbb{Z}_2^k \times \mathbb{Z}_2 = \mathbb{Z}_2^{k+1}$. Hence, $B_{k+1}$ is a permutation of $A_{k+1}$ and the proof is complete.
\end{proof}

\subsection*{The generating function of the trajectory $\{T^k(n)\}_{k=0}^\infty$}
Given a positive integer $n$, we are now interested in the generating function of the corresponding trajectory $\{T^k(n)\}_{k=0}^\infty$, that is the power series $\sum_{n=1}^{\infty} T^k(n) \cdot x^n$. The classical rational-transcendental dichotomy asserts that if a power series with integer coefficients converges in the unit disc, then either it defines a rational function or it admits the unit circle as a natural boundary (see \cite{Fatou} and \cite{Carl}). Using Corollary \ref{cor.Collatz-function}, we can prove that the generating function is a rational one and it has poles at the $2^k$-th roots of unity.

\begin{theorem}
For any $k=1,2,\ldots$, the series $\sum_{n=1}^{\infty} T^k(n) \cdot x^n$ converges for any $x$ with $\abs{x}<1$ and the sum is given by the following rational function:
$$\sum_{n=1}^{\infty} T^k(n) \cdot x^n = \frac{\sum_{\upsilon=1}^{2^k-1}T^k(\upsilon) x^\upsilon+x^{2^k} + \sum_{\upsilon=1}^{2^k-1}\overline{T^k(\upsilon)} x^{2^k+\upsilon}}{(1-x^{2^k})^2},$$
where $\overline{T^k(\upsilon)} =3^{d_{\upsilon k}} -T^k(\upsilon)$ and $\overline{T^k(\upsilon)}>0$ for all $k\in\mathbb{N}$ and $\upsilon=1,2,\ldots, 2^k-1$.
\end{theorem}

\begin{proof}
Firstly, we write:
$$\sum_{n=1}^{\infty} T^k(n) \cdot x^n = \sum_{\upsilon=0}^{2^k-1}\left(\sum_{n=\upsilon\mod2^k} T^k(n) \cdot x^n \right) = \sum_{\upsilon=0}^{2^k-1} A^k_\upsilon (x).$$
Therefore, it suffices to find the sum of the series $A^k_\upsilon (x)$ for any $\upsilon = 0,1,2,\ldots , 2^k-1$.

We start with the case where $\upsilon=0$. Then, by Corollary \ref{cor.Collatz-function}, we obtain:
$$
  A^k_0(x)=  \sum_{n=2^k \cdot p, p=1,2,\ldots} T^k(n) \cdot x^n = \sum_{p=1}^{\infty} p \cdot \left(x^{2^k}\right)^p = \frac{x^{2^k}}{(1-x^{2^k})^2},
$$
for any $x$ with $\abs{x}<1$.

Similarly, for $\upsilon>0$, by Corollary \ref{cor.Collatz-function} we obtain:
\begin{align*}
  A^k_\upsilon (n) =&   \sum_{n=2^k \cdot p+\upsilon, p=1,2,\ldots} T^k(n) \cdot x^n = \sum_{p=0}^{\infty} \left(3^{d_{\upsilon k}} \cdot p +T^k(\upsilon) \right) \cdot x^{2^k \cdot p +\upsilon}\\
  = &  x^\upsilon \cdot \sum_{p=0}^{\infty} \left(3^{d_{\upsilon k}}  p +T^k(\upsilon) \right) \cdot \left(x^{2^k}\right)^p\\
 =& x^\upsilon \cdot \frac{(3^{d_{\upsilon k}}-T^k(\upsilon)) x^{2^k} + T^k(\upsilon)}{(1-x^{2^k})^2}
\end{align*}
for any $x\in\mathbb{R}$ with $\abs{x}<1$. Setting $\overline{T^k(\upsilon)}= 3^{d_{\upsilon k}}-T^k(\upsilon)$, we obtain:
$$ A^k_\upsilon (n) = x^\upsilon \cdot \frac{\overline{T^k(\upsilon)} x^{2^k} + T^k(\upsilon)}{(1-x^{2^k})^2} = \frac{T^k(\upsilon) x^\upsilon +\overline{T^k(\upsilon)} x^{2^k+\upsilon}}{(1-x^{2^k})^2}.$$
It follows that the series $\sum_{n=1}^{\infty} T^k(n) \cdot x^n$ converges for any $x\in\mathbb{R}$ with $\abs{x} <1$ and its sum is given by:
\begin{align*}
  \sum_{n=1}^{\infty} T^k(n) \cdot x^n = & \sum_{\upsilon=0}^{2^k-1} A^k_\upsilon (x) \\
  = & \frac{x^{2^k}}{(1-x^{2^k})^2}+ \sum_{\upsilon=1}^{2^k-1} \frac{T^k(\upsilon) x^\upsilon +\overline{T^k(\upsilon)} x^{2^k+\upsilon}}{(1-x^{2^k})^2}\\
=& \frac{x^{2^k}}{(1-x^{2^k})^2}+  \frac{\sum_{\upsilon=1}^{2^k-1}\left(T^k(\upsilon) x^\upsilon +\overline{T^k(\upsilon)} x^{2^k+\upsilon}\right)}{(1-x^{2^k})^2} \\
=& \frac{\sum_{\upsilon=1}^{2^k-1}T^k(\upsilon) x^\upsilon  +x^{2^k} + \sum_{\upsilon=1}^{2^k-1}\overline{T^k(\upsilon)} x^{2^k+\upsilon}}{(1-x^{2^k})^2}.
\end{align*}

Finally, by induction on $k$ we show that $\overline{T^k(\upsilon)} = 3^{d_{\upsilon k}} - T^k(\upsilon) >0$ for all $k\ge 1$ and $\upsilon =0,1,\ldots, 2^k-1$. Indeed, for $k=1$ and $\upsilon=0 ,1$, the result is clear. Assume that the desired property holds true for some positive integer $k$ and for all $\upsilon =0,1,\ldots , 2^k-1$. We will prove the result for $k+1$. We distinguish two cases.

\begin{description}
  \item[Case I] If $T^k(\upsilon)$ is even, then $T^{k+1}(\upsilon) = \frac{T^k(\upsilon)}{2}$. We also have $x_{\upsilon,k+1}=0$ and thus $d_{\upsilon,k+1}=d_{nk}$. Consequently, we have to show that
      $$T^k(\upsilon) < 2\cdot 3^{d_{\upsilon k}} \quad \forall \upsilon=0,1,\ldots, 2^{k+1}-1.$$
      For $\upsilon=0,1,\ldots, 2^k-1$ the above inequality follows immediately by the inductive hypothesis. For $\upsilon=2^k, 2^k+1, \ldots, 2^{k+1}-1$, we write $\upsilon = 2^k + b$, where $0\le b \le 2^k-1$. Then $d_{\upsilon k}= d_{bk}$ (by Proposition \ref{prop.structure-x_kn}). By Corollary \ref{cor.Collatz-function} and the inductive hypothesis, we obtain
      $$T^k(n) = 3^{d_{bk}} +T^k(b) < 2\cdot 3^{d_{bk}} =2\cdot 3^{d_{\upsilon k}}.$$
  \item[Case II] If $T^k(\upsilon)$ is odd, then $T^{k+1}(\upsilon) = \frac{3T^k(\upsilon)+1}{2}$. We also have $x_{\upsilon,k+1}=1$ and thus $d_{\upsilon,k+1}=d_{nk}+1$. Consequently, we have to show that
      $$\frac{3T^k(\upsilon)+1}{2} < 3\cdot 3^{d_{\upsilon k}} \quad \forall \upsilon=0,1,\ldots, 2^{k+1}-1.$$
      For $\upsilon=0,1,\ldots, 2^k-1$ the above inequality follows immediately by the inductive hypothesis. For $\upsilon=2^k, 2^k+1, \ldots, 2^{k+1}-1$, we write $\upsilon = 2^k + b$, where $0\le b \le 2^k-1$. Then $d_{\upsilon k}= d_{bk}$ (by Proposition \ref{prop.structure-x_kn}). By Corollary \ref{cor.Collatz-function} and the inductive hypothesis, we obtain
      $$T^k(n) = 3^{d_{bk}} +T^k(b) \le 2\cdot 3^{d_{bk}} -1.$$
      Hence.
      $$\frac{3T^k(\upsilon)+1}{2} \le 3\cdot 3^{d_{bk}} -1 < 3\cdot 3^{d_{\upsilon k}}.$$
\end{description}
\end{proof}

\begin{example}
  For $k=1$, $k=2$ we have respectively:
$$\sum_{n=1}^{\infty} T(n) \cdot x^n = \frac{2x+x^2 +x^3}{(1-x^2)^2}$$
$$\sum_{n=1}^{\infty} T^2(n) \cdot x^n = \frac{x+2x^2 +8x^3+x^4 +2x^5 +x^6 +x^7}{(1-x^4)^2}.$$
\end{example}

\section{A review of density results}\label{sec.density_theorems}
The Collatz conjecture is equivalent to the statement that for any integer $n>1$ there exists $k$ (depending on $n$) such that $T^k(n)<n$. In this setting, Terras \cite{terras} and Everett \cite{ever} proved that the above statement is true for almost every positive integer $n$. More precisely, the set $M=\{n\in \mathbb{N}^* \mid (\exists k)(T^k(n)<n) \}$ has asymptotic density $1$, i.e.
$$\lim_{N\to \infty} \frac{\sharp\{n\in M \mid n<N\}}{N} =1.$$
Allouche \cite{allou} strengthened this result by showing that the set $M_\theta = \{n\in \mathbb{N}^* \mid (\exists k)(T^k(n)<n^\theta) \}$ has asymptotic density $1$, i.e. for almost every $n$ there is $k$ such that $T^k(n) < n^\theta$, for any fixed constant $\theta >0.869$. Furthermore, Korec \cite{kor} showed that the previous result remains valid for any constant $\theta > \frac{\log 3}{\log 4} \approx 0.7924$. Finally, Tao proved that if $f\colon \mathbb{N}^* \to \mathbb{R}$ is any function with $\lim_{N\to \infty} f(N) = \infty$, then the set $M_f = \{n\in \mathbb{N}^* \mid (\exists k)(T^k(n)<f(n)) \}$ has logarithmic density $1$ (see \cite{tao} for details).

For any positive integers $n,k$ we set
$$S_k(n) = x_{n0} + x_{n1} + \ldots +x_{n,k-1} = \textrm{card} \left\{ i\le k-1 \mid x_{ni} =1 \right\}.$$
By Theorem \ref{th.rows_of parity matrix}, the rows of the matrix $B_k$ form a permutation of the elements of $\mathbb{Z}_2^k$. Hence, it follows immediately that for any $d>0$ one has:
\begin{align*}
  \mathbb{P} \left(\left\{n\le 2^k \mid S_k(n) \le d k\right\} \right) = & \frac{1}{2^k} \left( \binom{k}{0} +\binom{k}{1} + \ldots +\binom{k}{[dk]} \right) \\
  = & \sum_{j=0}^{[dk]} \binom{k}{j} \frac{1}{2^j} \cdot \frac{1}{2^{k-j}}\\
=& F\left([dk] ; k, \frac{1}{2} \right),
\end{align*}
where $F(\cdot ; k, \frac{1}{2})$ stands for the cumulative distribution function of the binomial distribution with parameters $k$ and $p=\frac{1}{2}$ and $[dk]$ is the greatest integer less than or equal to $dk$.

\begin{lemma}
  For any real number $d\in \left(\frac{1}{2}, 1\right)$ we have
  $$\lim_{k\to \infty} \mathbb{P} \left(\left\{n\le 2^k \mid S_k(n) \le d k\right\} \right) =1.$$
\end{lemma}

\begin{proof}
  We will use the following inequality for the cumulative distribution function $F$ of the binomial distribution: if $\frac{1}{2} <\frac{n}{k} <1$, then
  $$\textrm{Pr} (X\ge n) =F\left(k-n ; k, \frac{1}{2} \right) \le \exp \left( -k D\left(\frac{n}{k} \| \frac{1}{2} \right) \right)$$
  where $D( a \| p)$ is the relative entropy:
  $$D(a\|p) = a\cdot \log\frac{a}{p} +(1-a) \cdot \log\frac{1-a}{1-p}.$$
  Since $\frac{[dk]+1}{k} > \frac{dk}{k} =d >\frac{1}{2}$, we obtain:
  \begin{align*}
    \textrm{Pr} \left(X >[dk]\right) \le & \textrm{Pr} \left(X \ge [dk]+1\right) \\
    \le  & \exp \left(-k D\left(\frac{[dk]+1}{k} \| \frac{1}{2} \right) \right).
  \end{align*}
  We have
  $$D\left(\frac{[dk]+1}{k} \| \frac{1}{2} \right) = \frac{[dk]+1}{k} \log\left(2 \frac{[dk]+1}{k}\right) + \left(1- \frac{[dk]+1}{k}\right) \log \left( 2 (1-\frac{[dk]+1}{k}) \right).$$
  As $\lim_{k\to \infty} \frac{[dk]+1}{k} =d$ it follows that
  $$\lim_{k\to \infty} D\left(\frac{[dk]+1}{k} \| \frac{1}{2} \right) = d \log 2d +(1-d) \log 2(1-d) >0 \text{ for all } d\in \left(\frac{1}{2},1\right). $$
  Hence,
  $$\lim_{k\to \infty} \exp \left(-k D\left(\frac{[dk]+1}{k} \| \frac{1}{2} \right) \right) = 0,$$
  and the result follows.
\end{proof}

\section{Collatz-Koopman operators}\label{sec-Koopman_operators}
In this section, we introduce two types of Koopman operators related to the Collatz problem.

\subsection*{The forward Collatz-Koopman operators.} The first one is a composition operator, where we choose as function space of observables (on the state space $\mathbb{N}$) the classical Banach sequence space $\ell_p$, $1\le p\le \infty$. By the definition of the Koopman operator (see Section \ref{sec.introduction}), this operator, denoted by $L_{T,p}$ is defined as follows:
$$ L_{T,p} \colon \ell_p \to \ell_p$$
$$L_{T,p}(x) = x\circ T.$$
Some equivalent and useful reformulations of the operator are described below. First of all, we observe that the map $T\colon \mathbb{N}^* \to \mathbb{N}^*$ is onto and more precisely we have the next lemma whose proof is straightforward.
\begin{lemma}\label{lemma-T-onto}
The $3x+1$-map $T\colon \mathbb{N}^*\to\mathbb{N}^*$ is onto and the following hold:
\begin{enumerate}
    \item for any $k\in\mathbb{N}^*$, $T(i)=3k \Leftrightarrow i=6k$;
    \item for any $k\in\mathbb{N}$, $T(i)=3k+1 \Leftrightarrow i=6k+2$;
    \item for any $k\in\mathbb{N}$, $T(i)=3k+2 \Leftrightarrow i=6k+4$ or $i=2k+1$.
\end{enumerate}
\end{lemma}
Consequently, an alternative definition of $L_{T,p}$ is the following: for every $x=(x_n)\in \ell_p$ one has:
$$L_{T,p} (x) =L_{T,p} \left( (x_n)_{n=1}^\infty \right) = (a_n)_{n=1}^\infty ,$$ where for any $k\in \mathbb{N}^*$,
$$a_{3k} = x_{6k}, \quad a_{3k+1} =x_{6k+2} \quad \text{and} \quad a_{3k+2}= x_{6k+4}+x_{2k+1}.$$

Furthermore, in the case $p\in[1,\infty)$, we have that $x=(x_n)_{n=1}^\infty =\sum_{n=1}^\infty x_n e_n \in \ell_p$, where $(e_n)$ stands for the usual vector basis of $\ell_p$. Then the above operator can be given by the formula:
\begin{align*}
  L_{T,p} (x) = L_{T,p} \left( \sum_{n=1}^\infty x_n e_n \right) & = \sum_{n=1}^\infty x_n e_{T(n)} \\
   & = \sum_{k=1}^{\infty} x_{6k} e_{3k} +\sum_{k=0}^{\infty} x_{6k+2} e_{3k+1} + \sum_{k=0}^{\infty} (x_{6k+4}+x_{2k+1})e_{3k+2}.
\end{align*}
Alternatively, the action of the operator on the basic vectors can be described by the relation $L_{T,p}(e_n) = e_{T(n)}$, and then $L_{T,p}$ can be extended linearly and continuously to give an operator to the whole space $\ell_p$.

Finally, for sequence spaces having a Schauder basis, bounded linear operators have a useful representation by an infinite matrix. For the operator $L_{T,p}$, $1\le p <\infty$, the corresponding matrix is the adjacency matrix of the forward Collatz graph (see Figure 1).

\begin{figure}[H]
\centering
\includegraphics[width=8cm,height= 4cm]{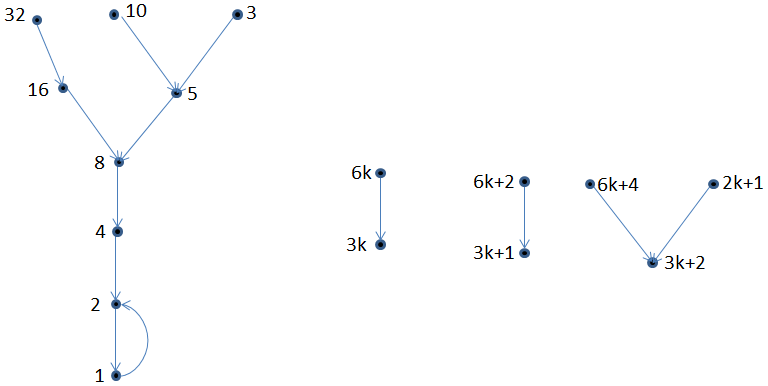}
\caption{Some parts of the forward Collatz graph.}
\end{figure}

The next result establishes that the above operators are well-defined (i.e. $L_{T,p}(x)\in\ell_p$ for any $x\in \ell_p$) and bounded and it also gives the exact value of their norm.

\begin{theorem}\label{th.norm-forward}
  The operators $L_{T,p}\colon\ell_p\to\ell_p$, $1\le p\le\infty$ are bounded and furthermore the norm of $L_{T,p}$ is given by $\|L_{T,p}\| = 2^{1-\frac{1}{p}}$. (The formula is valid in the case $p=\infty$, where we have $\|L_{T,\infty}\|=2$).
\end{theorem}

\begin{proof}
  The case $p=\infty$ is straightforward. For $p\in[1,\infty)$, we use the inequality $(a+b)^p \le 2^{p-1} (a^p+b^p)$ ($a,b\ge0$) to obtain:
  \begin{align*}
    \|L_{T,p}(x)\|^p  = & \Big\|\sum_{k=1}^{\infty} x_{6k} e_{3k} + \sum_{k=0}^{\infty} x_{6k+2} e_{3k+1}  +\sum_{k=0}^{\infty} (x_{6k+4}+x_{2k+1})e_{3k+2}\Big\| ^p  \\
    = & \sum_{k=1}^{\infty} \abs{x_{6k}}^p + \sum_{k=0}^{\infty} \abs{x_{6k+2}}^p + \sum_{k=0}^{\infty} \abs{x_{6k+4}+x_{2k+1}}^p \\
    \le & \sum_{k=1}^{\infty} \abs{x_{6k}}^p + \sum_{k=0}^{\infty} \abs{x_{6k+2}}^p +2^{p-1} \sum_{k=0}^{\infty} \abs{x_{6k+4}}^p+ 2^{p-1} \sum_{k=0}^{\infty} \abs{x_{2k+1}}^p\\
    \le & 2^{p-1} \|x\|^p.
  \end{align*}
  It follows that $\|L_{T,p}\| \le 2^{1-\frac{1}{p}}$. In order to establish the reverse inequality, we consider the vector $x=e_1+e_4$. Then, $L_{T,p}(x)=e_{T(1)}+e_{T(4)}=2e_2$. Thus,
  $$\|L_{T,p}\| \ge \frac{1}{\|x\|} \|L_{T,p}(x)\| = \frac{2}{2^{1/p}} = 2^{1-\frac{1}{p}}.$$

\end{proof}

\subsection*{The backward Collatz-Koopman operators.} Similarly, the second type of Collatz-Koopman operators, denoted by $B_{T,p}\colon\ell_p\to\ell_p$ ($p\in[1,\infty)$), is defined so that the corresponding matrix representation coincides with the adjacency matrix of the backward Collatz map. This implies that:
$$B_{T,p}(e_{3k}) = e_{6k} \,\, (k\ge 1), \quad B_{T,p}(e_{3k+1})=e_{6k+2} \quad \text{and} \quad B_{T,p}(e_{3k+2})=e_{6k+4}+e_{2k+1} \,\, (k\ge 0).$$
Therefore, for any $x=(x_n)_{n=1}^\infty \in \ell_p$ we have
$$B_{T,p}(x)= \sum_{k=1}^{\infty} x_{3k} e_{6k} +\sum_{k=0}^{\infty} x_{3k+1}e_{6k+2} + \sum_{k=0}^{\infty} x_{3k+2}e_{6k+4}+\sum_{k=0}^{\infty} x_{3k+2} e_{2k+1}.$$
Equivalently, $B_{T,p}(x)=B_{T,p}((x_n)_{n=1}^\infty)=(a_n)_{n=1}^\infty$, where,
$$a_{6k}=x_{3k}, \quad a_{6k+2}=x_{3k+1}, \quad a_{6k+4}=a_{2k+1}=x_{3k+2}.$$
The last equation makes also sense in the case of $\ell_\infty$, so the operators $B_{T,p}$ are defined for any $p\in[1,\infty]$.

The backward operators are closely related to the forward ones, since they coincide with their dual operators.

\begin{theorem}
    For any $p\in[1,\infty)$, we have that $L^*_{T,p}= B_{T,q}$ where $q$ is the conjugate index of $p$. In particular, $L^*_{T,2} = B_{T,2}$ and $\|B_{T,p}\| = 2^{\frac{1}{p}}$ for any $p\in [1,\infty]$.
\end{theorem}

\begin{proof}
  Firstly, it is easy to see that $B_{T,p} \colon \ell_p\to \ell_p$ is bounded for any $p\in [1,\infty]$. Furthermore, for any $k=0,1,2,\ldots$ and any $x=(x_n)\in \ell_q$ we observe that:
  \begin{align*}
    L^*_{T,p}(e_{3k+2})(x) = & e_{3k+2}\circ L_{T,p} (x) = e_{3k+2} \Big(\sum_{n=1}^{\infty} x_n e_{T(n)}\Big) \\
    = & \sum_{n: T(n)=3k+2} x_n = x_{6k+4} + x_{2k+1} = e_{6k+4}(x) +e_{2k+1}(x)
  \end{align*}
  Thus, $L^*_{T,p}(e_{3k+2})= e_{6k+4} +e_{2k+1}$. Similarly, we can show that $L^*_{T,p} (e_{3k})= e_{6k}$, for any $k\ge 0$, and $L^*_{T,p} (e_{3k+1})= e_{6k+2}$, for any $k\ge 1$. By linearity and boundedness of $B_{T,q}, L^*_{T,p} \colon \ell_q \to \ell_q $, it follows that $L^*_{T,p}$ coincides with $B_{T,q}$.

  Finally, for $p=1$, the operator $B_{T,1}$ is not the dual of $L_{T,\infty}$. However, we can easily prove that the norm of the operator is equal to $2$. Indeed, for any $x=(x_n)\in\ell_1$,
  \begin{align*}
    \|B_{T,1}(x)\| = & \Big\| \sum_{k=1}^{\infty} x_{3k} e_{6k} +\sum_{k=0}^{\infty} x_{3k+1}e_{6k+2} + \sum_{k=0}^{\infty} x_{3k+2}e_{6k+4}+\sum_{k=0}^{\infty} x_{3k+2} e_{2k+1} \Big\| \\
    = & \sum_{k=0}^{\infty} \abs{x_{3k}} +\sum_{k=1}^{\infty} \abs{x_{3k+1}} + \sum_{k=1}^{\infty} \abs{x_{3k+2}}+\sum_{k=1}^{\infty} \abs{x_{3k+2}}  \\
    \le & 2\|x\|.
  \end{align*}
  Therefore, $\|B_{T,1}\| \le 2$. For the reverse inequality, we take the vector $x=e_2$ for which it holds $B_{T,1}(x)=e_1+e_4$. Hence,
  $$\|B_{T,1}\| \ge \frac{1}{\|x\|} \|B_{T,1}(x)\|=\|e_1+e_4\|=2$$
  and the result follows.
  \end{proof}

\section{Spectral radius of the forward and backward Collatz-Koopman operators}\label{sec. spectral radius}
This section contains an estimate of the spectral radius of the Collatz-Koopman operators defined in the previous section. In order to obtain this estimate, we need a series of preliminary results.

\begin{lemma}\label{lemma-number}
  For every $n\in\mathbb{N^*}$ the numbers $p=(2^n-1) \mod 2^n$ satisfy the following properties:
  \begin{enumerate}
    \item [(i)] $T^k(p)= 1 \mod 2$ for every $k=0,1,\ldots , n-1$.
    \item [(ii)] $T^k(p) = 2\mod3$ for every $k=1,2,\ldots, n$.
  \end{enumerate}
\end{lemma}

\begin{proof}
  Let $p=2^n\cdot \lambda +2^n-1$, for some $\lambda \in \mathbb{N}$. Using induction, one can show that for every $k=0,1,2,\ldots,n$,
  $$T^k(p) = 3^k 2^{n-k} \lambda + 3^k 2^{n-k} -1.$$
  The desired properties follow easily by the above equation.
\end{proof}

For every $n\in \mathbb{N}$ and every positive integer $m$, we set
$$A^n_m = \{i\in \mathbb{N}^* \mid T^n(i)=m \}.$$
Since $T^n$ in onto for every $n$, we obtain that $A^n_m$ are non empty sets for every $n\in \mathbb{N}$, $m\in\mathbb{N}^*$. Furthermore, since $T^0$ is the identity map on $\mathbb{N}^*$, we have that $\abs{A^0_m} =1$ for every positive integer $m$.  We next summarize the basic properties of the sets $(A^n_m)_{n,m}$.

\begin{proposition}\label{prop.set_A^n_m}
  The following properties hold.
  \begin{enumerate}
    \item [(i)] For every $n\in\mathbb{N}$, the sets $(A^n_m)_{m=1}^\infty$ are pairwise disjoint.
    \item [(ii)] $A^n_m = A^{n-1}_{2m}$ for any $m=0 \mod 3$ or $m=1\mod 3$; $A_m^n = A^{n-1}_{2m} \cup A^{n-1}_{2k+1}$ for every $m =2 \mod 3$, where $m=2k+3$.
    \item [(iii)] $\abs{A^n_m} \le \abs{A^{n+1}_m}$ for every $n,m$.
    \item [(iv)] $A^n_m$ is a finite set, for every $n,m$.
  \end{enumerate}
\end{proposition}

\begin{proof}
The first assertion is quite clear. For the second one, we observe that
$$A^n_m = \{i\in \mathbb{N}^* \mid T^n(i) = m \} =\{i\in \mathbb{N}^* \mid T(T^{n-1}(i))=m\}.$$
The result now follows easily from Lemma \ref{lemma-T-onto}. As far as the third assertion, let $i\in\mathbb{N}^*$ be such that $T^n(i)=m$. Then,
$$T^{n+1}(2i) = T^n(T(2i))=T^n(i)=m.$$
Hence, the map $A^n_m \to A^{n+1}_m$ with $i\mapsto 2i$ defines a one-to-one correspondence from $A^n_m$ into $A^{n+1}_m$. Finally, for the fourth assertion, we have that $\abs{A^1_m} \le 2$ for any $m\in\mathbb{N}^*$. Using the second assertion and induction on $n\in\mathbb{N}^*$, we can easily prove that $A^n_m$ is finite for every $n,m$.
\end{proof}

For every $n\in\mathbb{N}$, we let $c_n$ be the maximal cardinality of the sets $(A^n_m)_{m=1}^\infty$. That is
$$c_n = \max_{m\in\mathbb{N}^*} \abs{A^n_m}.$$
We also consider the Fibonacci sequence $(F_n)_n$ with $F_0=1$ and $F_1=2$. Hence, for every $n\ge 2$ we have $F_n=F_{n-1}+F_{n-2}$. Then, the next proposition is proven.

\begin{proposition}\label{prop.spectral}
  For every $n\in\mathbb{N}$, it holds $n\le c_n \le F_n$.
\end{proposition}

\begin{proof}
  For $n=0$ and $n=1$, it is clear that $c_0=F_0=1$ and $c_1=F_1=2$. We now proceed by induction on $n$. We assume that for any $0\le k\le k$, we have $k\le c_k \le F_k$. We show that $n\le c_n \le F_n$. We distinguish three cases.

  \begin{description}
    \item[Case I] Using induction on $n$, it is easy to see that $\abs{A^n_m}=1$ for every positive integer $m=0\mod 3$ and every $n\in\mathbb{N}$. Indeed, for every $m=0\mod 3$, we have $A^0_m=1$ and $\abs{A^n_m} = \abs{A^{n-1}_{2m}}$, where $2m=0\mod 3$. So, by the inductive hypothesis, it follows that $\abs{A^n_m}=1$.

    \item[Case II] If $m=1\mod 3$, then, by Proposition \ref{prop.set_A^n_m}, it follows that:
  $$\abs{A^n_m} = \abs{A^{n-1}_{2m}} \le \abs{A^n_{2m}}.$$
  We observe that $2m=2 \mod 3$, hence this case can be reduced to Case III below.

    \item[Case III] If $m=2\mod 3$ and $m=2k+1$, then by Proposition \ref{prop.set_A^n_m} we have that:
  $$ \abs{A^n_m} = \abs{A^{n-1}_{2m}}+\abs{A^{n-1}_{2k+1}} \le \abs{A^{n-1}_{2m}}+F_{n-1}.$$
  Since, $2m=1 \mod 3$, it follows that $A^{n-1}_{2m} = A^{n-2}_{4m}$, and hence $\abs{A^{n-1}_{2m}} = \abs{A^{n-2}{4m}} \le F_{n-2}$. Therefore, $\abs{A^n_m} \le F_{n-2}+ F_{n-1}=F_n$.
  \end{description}

  The three cases above imply that $\abs{A^n_m} \le F_n$ for every $m\in \mathbb{N}^*$. Thus, $c_n\le F_n$.

  Finally, using the numbers $p=(2^n-1)\mod 2^n$, given by Lemma \ref{lemma-number}, we can see that $\abs{A^n_{T^n(p)}}\ge n$, which completes the proof.
\end{proof}

We are now ready to obtain our estimation of the spectral radius of the forward Collatz-Koopman operators $L_{T,p}$.

\begin{theorem}\label{th.norm-Collatz^n}
  For any $n\in \mathbb{N}^*$, we have $ \|L^n_{T,p} \| =c_n^{1-\frac{1}{p}}$.
\end{theorem}

\begin{proof}
  First of all, we observe that for any $n\in\mathbb{N}^*$ we have $L^n_{T,p} = L_{T^n,p}$. Indeed, for any $x\in\ell_p$,
  $$ L^n_{T,p} = L_{T,p} \cdots L_{T,p}(x) = x\circ T\circ \cdots \circ T= L_{T^n,p}(x).$$
  Therefore, we have to calculate the norm of the operator $L_{T^n,p}$. The proof is now similar to the proof of Theorem \ref{th.norm-forward}. Indeed, it is easy to see that the operator is given by:
  $$ L_{T^n,p}(x)= L_{T^n,p}((x_i))= \sum_{m=1}^{\infty} \left(\sum_{i: T^n(i)=m} x_i\right) e_m = \sum_{m=1}^{\infty} \left(\sum_{i\in A^n_m} x_i\right) e_m.$$
  Therefore,
  \begin{align*}
    \|L_{T^n,p}(x)\|^p =  & \Big\|  \sum_{m=1}^{\infty} \Big(\sum_{i\in A^n_m}  x_i\Big) e_m \Big\| = \sum_{m=1}^{\infty} \Big|\sum_{i\in A^n_m} x_i \Big|^p \\
    \le  & \sum_{m=1}^{\infty} \abs{A^n_m}^{p-1} \Big(\sum_{i\in A^n_m} \abs{x_i}^p \Big) \\
    \le & c_n^{p-1} \sum_{m=1}^{\infty} \sum_{i\in A^n_m}\abs{x_i}^p = c_n^{p-1} \|x\|^p
  \end{align*}
where the last equality follows from the fact that $(A^n_m)_{m=1}^\infty$ are pairwise disjoint. Hence, $\|L_{T^n,p}\| \le c_n^{1-1/p}$. For the reverse inequality, we fix $m_0\in\mathbb{N}^*$ such that $\abs{A^n_{m_0}} = \max_{m\in\mathbb{N}^*} \abs{A^n_m} =c_n$ and we consider the vector $x=\sum_{i\in A^n_{m_0}} e_i$. Then, $\|x\| = \abs{A^n_{m_0}}^{1/p} = c_n^{1/p}$ and $L_{T^n,p}(x) = \abs{A^n_{m_0}} e_{m_0} = c_n e_{m_0}$. Thus,
$$ \|L_{T^n,p}\| \ge \frac{1}{\|x\|} \|L_{T^n,p}(x)\|=\frac{c_n}{c_n^{1/p}} = c_n^{1-1/p}$$
and the result follows.
\end{proof}

\begin{corollary}\label{cor.spectral_radius}
  For any $p\in[1,\infty]$, the spectral radius $\rho(L_{T,p})$ of the forward Collatz-Koopman operator $L_{T,p}$ satisfies the inequality
  $$1\le \rho(L_{T,p}) \le \varphi^{1-1/p},$$
  where $\varphi = \frac{1+\sqrt{5}}{2}$ stands for the golden ratio. In particular, for $p=1$ we obtain that $\rho(L_{T,1})=1$.
\end{corollary}

\begin{proof}
  The spectral radius of the operator $L_{T,p}$ is given by $\rho(L_{T,p}) = \lim_n \|L_{T,p}^n\|^{1/n}$. By Theorem \ref{th.norm-Collatz^n} and Proposition \ref{prop.spectral}, it follows immediately that
  $$ \left(n^{1/n}\right)^{1-1/p} \le \|L_{T,p}^n\|^{1/n} \le \left(F_n^{1/n}\right)^{1-1/p}.$$
  It is well-known that $\lim_n F_n^{1/n}=\varphi$ (the golden ratio) and the desired result follows.
\end{proof}

\begin{remark}
  The numbers of Lemma \ref{lemma-number} have the additional property that $ 2^{2i}T^k(p) = 1\mod 3$ for every $k=1,2,\ldots , n$ and $i=1,2,\ldots, [\frac{k}{2}]$. This remark allows us to obtain a slightly better lower estimate for $c_n$ in Proposition \ref{prop.spectral}. Namely, it is proved that $c_n \ge c_{n-1}+\left[\frac{n+1}{2} \right]$. Hence, $c_n\ge \frac{1}{4} (n+1)(n+2)$. However, this lower estimation is also polynomial, and thus it does not affect the result in Corollary \ref{cor.spectral_radius}.
\end{remark}

\begin{corollary}
  For any $p\in [1,\infty]$, the spectral radius of the backward Collatz-Koopman operator satisfies
  $$1 \le \rho(B_{T,p}) \le \varphi^{1/p}.$$
  In particular, for $p=\infty$ we obtain that $\rho(B_{T,\infty})=1$.
\end{corollary}

\begin{proof}
  For $p\in (1,\infty]$, since $B_{T,p}$ is the dual operator of $L_{T,q}$, where $q$ is the conjugate index of $p$, we have that $\rho(B_{T,p}) = \rho(L_{T,q}) $ and the result follows immediately. For $p=1$, it suffices to show that $\|B^n_{T,1}\| =c_n$ for any $n\in\mathbb{N}^*$.

  Indeed, for any $n\in\mathbb{N}^*$ and any $x=(x_m)\in \ell_1$, we have
  $$B^n_{T,1}(x) =\sum_{m=1}^{\infty} \sum_{i: \, T^n(i)=m} x_m e_i = \sum_{m=1}^{\infty} \sum_{i\in A^n_m} x_m e_i.$$
  Hence,
  \begin{align*}
    \|B^n_{T,1}(x)\| = & \Big\|\sum_{m=1}^{\infty} \sum_{i\in A^n_m} x_m e_i \Big\| = \sum_{m=1}^{\infty} \sum_{i\in A^n_m} \abs{x_m} \\
    \le & \sum_{m=1}^{\infty} c_n \abs{x_m} = c_n \|x\|.
  \end{align*}
  Thus $\| B^n_{T,1}\| \le c_n$. If $m_0 \in \mathbb{N}^*$ is such that $\abs{A^n_{m_0}} = \max_{m\in\mathbb{N}*} \abs{A^n_m}=c_n$, then we take the vector $x=e_{m_0}$ and we have that $B^n_{T,p} (x) = \sum_{i\in A^n_{m_0}} e_i$. Therefore,
  $$\|  B^n_{T,1}\| \ge \frac{1}{\|x\|} \| B^n_{T,1}(x)\| = c_n.$$
\end{proof}

\section{The Collatz-Koopman operator on $\ell_1$.}\label{sec.ell-1-operator}
In this section, we focus our attention on the Collatz-Koopman operator defined on $\ell_1$, i.e. $L_{T,1} \colon \ell_1 \to \ell_1$, given by
$$L_{T,1}(x) = x\circ T \quad \forall x\in \ell_1, \text{ or}$$
\begin{align*}
  L_{T,1}(x) = & \sum_{n=1}^{\infty} x_n e_{T(n)} = \sum_{m=1}^{\infty} \Big(\sum_{n\in A^1_m} x_n \Big) e_m \\
  = & \sum_{k=1}^{\infty} x_{6k} e_{3k} + \sum_{k=0}^{\infty} x_{6k+2} e_{3k+1} +\sum_{k=0}^{\infty} (x_{6k+4}+x_{2k+1}) e_{3k+2}.
\end{align*}
By Sections \ref{sec-Koopman_operators} and \ref{sec. spectral radius}, we know that $L_{T,1}$ is a bounded operator of norm one and its spectral radius is also $\rho(L_{T,1}) =1$.

The Collatz conjecture refers to the convergence of the sequence $(T^k(n))_{k=1}^\infty$ for every $n\in \mathbb{N}^*$. However, it is well known that for an arbitrary sequence $(a_n)_{n}$, the sequence of averages $\left( \frac{a_1+a_2+\ldots+a_n}{n} \right)$ has better convergence properties. Motivated by this remark, we define the average operator as follows. For any $x\in \ell_1$ we set
$$\mathcal{K}(x) = \lim_{k\to \infty} \frac{x+L_{T,1}(x) + \ldots +L_{T,1}^{k-1}(x)}{k}.$$
Unavoidably, we have to address the issue whether the operator $\mathcal{K}$ is well defined. The next lemma contains the fundamental information concerning the operator $L$. In order to state the lemma, we need a piece of notation. Define a cycle to be a (finite) set of positive integers
$$\mathcal{C} = \{ n_1 ,n_2 ,\ldots ,n_k\}$$
such that $T(n_i)=n_{i+1}$ for any $i=1,2,\ldots, k-1$, and $T(n_k)=n_1$. The \emph{length of the cycle} $\mathcal{C}$, denoted by $l(\mathcal{C})$, is the number $k$, that is the cardinality of the set $\mathcal{C}$. It is also clear that two cycles are disjoint sets. For instance, the set $\mathcal{C}_1=\{1,2\}$ is a cycle of length $2$ and we call it the \emph{trivial cycle}.

If the Collatz conjecture is true, then there is only one cycle, namely the trivial one. However, to the best of our knowledge, it is an open question whether there are finitely many cycles or not. Nevertheless, it is trivial that there are at most countably many cycles, which we can enumerate as $(\mathcal{C}_i)_{i=1}^\infty$ such that $\min\mathcal{C}_i <\min\mathcal{C}_j$.

For every cycle $\mathcal{C}_i$ we let $N_i$ be the subset of $\mathbb{N}^*$ containing all natural numbers $n$ whose trajectory $(T^k(n))_{k=0}^\infty$ ends in the cycle $\mathcal{C}_i$. Formally, we have
$$N_i = \{ n\in \mathbb{N}^* \mid (\exists k\in\mathbb{N}^*)[T^k(n) \in \mathcal{C}_i] \}.$$
We also set $N_\infty =\{ n\in \mathbb{N}^* \mid T^k(n) \text{ diverges to infinity}\}$. Then, $\{N_\infty, N_i \mid i\in \mathbb{N}^*\}$ defines a partition of $\mathbb{N}^*$ into pairwise disjoint sets.

Suppose now that there is $n\in\mathbb{N}^*$ such that the corresponding trajectory $\{n, T(n), T^2(n), \ldots \}$ of the Collatz map diverges to infinity. Then, for any $k\in\mathbb{N}$, we obtain
$$\frac{e_n+L_{T,1}(e_n) + \ldots +L_{T,1}^{k-1}(e_n)}{k} = \sum_{i=0}^{k-1} \frac{1}{k} e_{T^i(n)}.$$
It follows easily that $\left\| \frac{e_n+L_{T,1}(e_n) + \ldots +L_{T,1}^{k-1}(e_n)}{k} \right\|_{\ell_1} =1$ for any $k\in\mathbb{N}^*$, however the sequence of averages $\Big( \frac{e_n+L_{T,1}(e_n) + \ldots +L_{T,1}^{k-1}(e_n)}{k} \Big)_{k=1}^\infty$ converges pointwise to $0$. This example shows that the average operator ignores the divergent trajectories. However, this inconvenience occurs only in the case of divergent trajectories.

\begin{lemma}\label{lemma-average-operator}
  Assume that $N_\infty=\emptyset$, i.e. no divergent trajectories exist. Then for every $x=(x_n)_{n=1}^\infty$ the sequence $\left( \frac{x+L_{T,1}(x) + \ldots +L_{T,1}^{k-1}(x)}{k} \right)$ converges with respect to the $\ell_1$-norm topology and its limit is given by
$$ \lim_{k\to \infty} \frac{x+L_{T,1}(x) + \ldots +L_{T,1}^{k-1}(x)}{k} = \sum_{i=1}^{\infty} \left( \frac{\sum_{n\in  N_i}x_n}{l(\mathcal{C}_i)} \right) \cdot \sum_{j\in \mathcal{C}_i} e_j.$$
\end{lemma}

\begin{proof}
  Since $x$ belongs to $\ell_1$, it is easy to see that the series $\sum_{i=1}^{\infty} \left( \frac{\sum_{n\in  N_i}x_n}{l(\mathcal{C}_i)} \right) \cdot \sum_{j\in \mathcal{C}_i} e_j$ converges in $\ell_1$ and furthermore,
$$\| \sum_{i=1}^{\infty} \left( \frac{\sum_{n\in  N_i}x_n}{l(\mathcal{C}_i)} \right) \cdot \sum_{j\in \mathcal{C}_i} e_j\|_{\ell_1} \le \|x\|_{\ell_1}.$$

We prove the desired result progressively starting with the vectors $e_n$'s, then passing to finite linear combinations of them and finally for an arbitrary $x\in \ell_1$. So, assume first that $x=e_n$. Then $n$ belongs to $N_i$ for some $i\in \mathbb{N}^*$, which means that the trajectory of $n$ ends in the cycle $\mathcal{C}_i$. Let $\mathcal{C}_i = \{ n_1, n_2, \ldots , n_l\}$. We have to prove that
$$\lim_{k\to \infty} \frac{e_n+L_{T,1}(e_n) + \ldots +L_{T,1}^{k-1}(e_n)}{k} =  \frac{1}{l}  \cdot \sum_{j=1}^l e_{n_j}.$$
Indeed, let $k_0$ be the least natural number such that $T^{k_0}(n)=n_1$. Then for every $k>k_0$, we write $k-1-k_0=rl+\upsilon$, where $0\le \upsilon <l$, and we have
\begin{align*}
\frac{e_n+L_{T,1}(e_n) + \ldots +L_{T,1}^{k-1}(e_n)}{k}=&\frac{e_n+e_{T(n)} + \ldots +e_{T^{k-1}(n)}}{k}\\
=& \frac{\sum_{i=1}^{k_0} e_{T^i(n)}}{k} + \frac{(r+1) \sum_{j=1}^{\upsilon} e_{n_j}+r \sum_{j=\upsilon+1}^{l}e_{n_j}}{k_0+1+rl+\upsilon}.
\end{align*}
Letting $k$ tend to infinity (i.e. $r\to \infty$) we have the desired result.

Assume now that $x=(x_n)$ is a finitely supported sequence in $\ell_1$. Then $x$ is a finite linear combination of the basis vectors $e_n$'s, that is there is $N\in \mathbb{N}^*$ such that $x=\sum_{m=1}^{N} x_m e_m$. Therefore, we obtain
\begin{align*}
  \mathcal{K}(x) =  & \sum_{m=1}^{N} x_m \mathcal{K}(e_m) =  \sum_{i=1}^{\infty}\sum_{m\le N, m\in N_i} x_m \mathcal{K}(e_m) \\
  = & \sum_{i=1}^{\infty}\sum_{m\le N, m\in N_i} x_m \frac{1}{l_i}  \cdot \sum_{j\in \mathcal{C}_i} e_j=  \sum_{i=1}^{\infty} \frac{\sum_{m\le N, m\in N_i}x_m}{l_i}  \cdot \sum_{j\in \mathcal{C}_i} e_j,
\end{align*}
where every sum is a finite one and the desired result is clear in this case.

Finally, assume that $x=(x_n)$ is any vector of $\ell_1$. Then, for any $\epsilon>0$, we can consider $N\in \mathbb{N}$ large enough so that $\sum_{n=N}^{\infty} \abs{x_n} <\epsilon$ and set $y=(y_n)_{n=0}^\infty=(x_1, \ldots, x_N, 0,0, \ldots)$. Therefore, $\|x-y\| <\epsilon$. Furthermore, for the finitely supported vector $y$, we have verified the desired result. Hence, for all $k$ sufficiently large, we have:
$$\Big\| \frac{y+L_{T,1}(y) + \ldots +L_{T,1}^{k-1}(y)}{k} - \sum_{i=1}^{\infty} \Big( \frac{\sum_{n\in  N_i}y_n}{l(\mathcal{C}_i)} \Big) \cdot \sum_{j\in \mathcal{C}_i} e_j \Big\| \le \epsilon.$$
Therefore, for all sufficiently large $k$, we obtain
\begin{align*}
  \Big\|  \frac{x+L_{T,1}(x) + \ldots +L_{T,1}^{k-1}(x)}{k} & - \sum_{i=1}^{\infty} \Big( \frac{\sum_{n\in  N_i}x_n}{l(\mathcal{C}_i)} \Big) \cdot \sum_{j\in \mathcal{C}_i} e_j \Big\|  \le \\
 &\le   \Big\|\frac{x+L_{T,1}(x) + \ldots +L_{T,1}^{k-1}(x)}{k} - \frac{y+L_{T,1}(y) + \ldots +L_{T,1}^{k-1}(y)}{k} \Big\| \\
& +  \Big\| \frac{y+L_{T,1}(y) + \ldots +L_{T,1}^{k-1}(y)}{k} - \sum_{i=1}^{\infty} \Big( \frac{\sum_{n\in  N_i}y_n}{l(\mathcal{C}_i)} \Big) \cdot \sum_{j\in \mathcal{C}_i} e_j \Big\| \\
&+  \Big\| \sum_{i=1}^{\infty} \Big( \frac{\sum_{n\in  N_i}y_n}{l(\mathcal{C}_i)} \Big) \cdot \sum_{j\in \mathcal{C}_i} e_j - \sum_{i=1}^{\infty} \Big( \frac{\sum_{n\in  N_i}x_n}{l(\mathcal{C}_i)} \Big) \cdot \sum_{j\in \mathcal{C}_i} e_j \Big\|\\
& \le 2\epsilon +\Big\| \sum_{i=1}^{\infty} \frac{1}{l(\mathcal{C}_i)} \Big(\sum_{n\in  N_i, n>N}-x_n\Big) \cdot \sum_{j\in \mathcal{C}_i} e_j  \Big\|\\
& \le 2\epsilon + \sum_{i=1}^{\infty} \Big|\sum_{n\in  N_i, n>N}-x_n  \Big|\\
& \le 2\epsilon + \sum_{n>N}\abs{x_n} \le 3\epsilon.
\end{align*}

\end{proof}

The next result is now easy to establish.

\begin{theorem}\label{th-average-operator}
  Assume that the Collatz dynamical system has no divergent trajectories. Then the average operator $\mathcal{K}\colon \ell_1 \to \ell_1$ is well-defined and bounded with norm one, i.e. $\mathcal{K}$ belongs to the unit sphere of $\mathcal{L}(\ell_1)$.
\end{theorem}

\begin{proof}
  The previous lemma implies that $\mathcal{K}$ is well-defined. Since $\|L_{T,1}\|=1$, it is also easy to verify that $\|\mathcal{K}\|\le1$. Finally, by the previous lemma we obtain that $\mathcal{K}(e_1)=\frac{1}{2} e_1 +\frac{1}{2} e_2$. Hence,
  $$\|\mathcal{K} \| \ge \frac{1}{\|e_1\|} \|\mathcal{K}(e_1) \| = 1,$$
  and the result follows.
\end{proof}

\begin{remark}
  The above theorem can be described as an analogue of the von Neumann's mean ergodic theorem (see for example \cite{walters}).
\end{remark}

The next corollary is straightforward.

\begin{corollary}\label{cor-average-operator}
The following are equivalent.
\begin{enumerate}
  \item The Collatz conjecture is true.
  \item For any $x=(x_i)_{i\in\mathbb{N}^*}\in \ell_1$, we have:
$$\mathcal{K}(x) = \frac{\sum_{i=1}^{\infty}x_i}{2} e_1 + \frac{\sum_{i=1}^{\infty}x_i}{2} e_2.$$
 \item $\mathcal{K}$ is a rank-2 operator.
\end{enumerate}
\end{corollary}

In a similar way, we can study the operator $\mathcal{L}$ defined by
$$\mathcal{L}(x) = \lim_{k\to \infty} \frac{x-L_{T,1}(x) + \ldots +(-1)^{k-1} L_{T,1}^{k-1}(x)}{k}.$$
In order to formulate the corresponding result, we need some additional notation. Firstly, for every cycle $\mathcal{C}_i$, $i=1,2,\ldots$, we fix an enumeration of its elements, i.e.
$$\mathcal{C}_i =\{n_{ji} \mid j=1,2,\ldots , l(\mathcal{C}_i) \},$$
where $T(n_{ji})=n_{j+1,i}$ for $j=1,2,\ldots, l(\mathcal{C}_i)-1$, and $T(n_{l(\mathcal{C}_i)i})=n_{1i}$. Furthermore, for every $n\in N_i$, we denote by $s(n)$ the least integer $k\ge 0$ such that $T^{s(n)}(n)=n_{1i}$. Then, we have the next result whose proof is similar to the proof of Theorem \ref{th-average-operator} (and Lemma \ref{lemma-average-operator}) and it is omitted.

\begin{theorem}
  Assume that $N_\infty = \emptyset$, i.e. no divergent trajectories exist. Then the operator $\mathcal{L} \colon \ell_1\to \ell_1$ given by
  $$ \mathcal{L}(x) = \lim_{k\to \infty} \frac{x-L_{T,1}(x) + \ldots +(-1)^{k-1} L_{T,1}^{k-1}(x)}{k},$$
  is well-defined and bounded with norm $1$. Furthermore, for every $x=(x_n)_{n=1}^\infty \in \ell_1$, we obtain:
  $$ \mathcal{L}(x) =  \sum_{l(\mathcal{C}_i) \colon \text{even}} \frac{\sum_{n\in N_i} (-1)^{s(n)} x_n}{l(\mathcal{C}_i)}  \sum_{j=1}^{l(\mathcal{C}_i)} (-1)^{j-1} e_{n_{ji}}.$$
\end{theorem}

Finally, the next corollary can be proved.

\begin{corollary}
  The following are equivalent.
  \begin{enumerate}
  \item The only cycle with even length is the trivial one.
  \item For any $x=(x_i)_{i\in\mathbb{N}^*}\in \ell_1$, we have:
$$\mathcal{L}(x) = \frac{\sum_{i=1}^{\infty}(-1)^{s(i)}x_i}{2} e_1 + \frac{\sum_{i=1}^{\infty}(-1)^{s(i)}x_i}{2} e_2,$$
  where $s(i)$ is the least integer such that $T^{s(i)}(i) =1$.
 \item $\mathcal{L}$ is a rank-2 operator.
\end{enumerate}
\end{corollary}

\subsection*{Spectral properties of $L_{T,1}$}
One of the main advantages in associating a linear operator with a dynamical system is that the behaviour of the system, for instance trajectories, attractors, fixed points etc, can be translated into spectral objects (i.e. eigenvalues, eigenvectors). This general remark can also be applied to our setting.

We associate with each cycle $\mathcal{C}_i$ the vector $u_i = \sum_{n\in\mathcal{C}_i} e_n$. The cycles, through the previous vectors, are closely related to the eigenspace $E_1$ of the eigenvalue $\lambda=1$.

\begin{theorem}\label{th. eigenspace}
The eigenspace of the eigenvalue $\lambda=1$ of $L_{T,1}$ coincides with the closed linear span of the vectors $(u_i)_{i=1}^\infty$.
\end{theorem}

\begin{proof}
Assume that $\mathcal{C}_i$ is the cycle:
$$\mathcal{C}_i =\{n_1,n_2,\ldots ,n_k\},$$
where $T(n_i)=n_{i+1}$ for $i=1,2,\ldots, k-1$, and $T(n_k)=n_1$. Then $u_i = e_{n_1}+e_{n_2}+\ldots e_{n_k}$ and it is easy to observe that $L_{T,1}(u_i)=u_i$. Therefore, $\overline{\textrm{span}}\{u_i \mid i\in \mathbb{N}^*\} \subset E_1$.

For the inverse inclusion, assume that $v\in E_1$ is an eigenvector of the eigenvalue $\lambda=1$. Observe that $\mathcal{K}(v)=v$. Therefore, Lemma \ref{lemma-average-operator} implies that
$$v=\sum_{i=1}^{\infty} \left( \frac{\sum_{n\in  N_i}v_n}{l(\mathcal{C}_i)} \right) \cdot \sum_{j\in \mathcal{C}_i} e_j = \sum_{i=1}^{\infty} \left( \frac{\sum_{n\in  N_i}v_n}{l(\mathcal{C}_i)} \right) \cdot u_i \in \overline{\textrm{span}}\{u_i \mid i\in \mathbb{N}^*\}.$$
\end{proof}

\begin{corollary}
\begin{enumerate}
    \item The $3x+1$-problem is equivalent to the following statement: The eigenspace of the eigenvalue $\lambda=1$ for $L_{T,1}$ is a one dimensional vector space.

    \item There are finitely many cycles if and only if the eigenspace $E_1$ is finitely dimensional.
\end{enumerate}
\end{corollary}

In a similar way, the dimension of the eigenspace $E_{-1}$ of the eigenvalue $\lambda=-1$ is related to the number of cycles of even length. Indeed, we associate with each cycle $\mathcal{C}_i = \{n_1, n_2, \ldots, l\}$ of even length, the vector $v_i = \sum_{j=1}^{l} (-1)^{j-1} e_{n_j}$. Using the operator $\mathcal{L}$, we can prove (as in Theorem \ref{th. eigenspace}), the next result.

\begin{theorem}
  The eigenspace $E_{-1}$ of $\lambda=-1$ coincides with the closed linear span of the vectors $\{ v_i \mid \mathcal{C}_i \text{ has even length} \}$.
\end{theorem}

\begin{corollary}
  There are finitely many cycles if and only if the eigenspace $E_{-1}$ is finitely dimensional.
\end{corollary}

\section{The sequence $(-1)^{T^k(n)}$: periodicity and frequency content}\label{sec.periodicity}
In order to proceed with the investigation of the Collatz problem, we focus now on the sequence of signs $\{(-1)^{T^k(n)}\}_{k,n=1}^\infty$. This sequence also plays a central role in the $3x+1$-problem, as the next proposition shows.

\begin{proposition}
  Let $n\in \mathbb{N}^*$ be fixed and let $a_k=(-1)^{T^k(n)}$ for any $k\in\mathbb{N}$. Then, the following hold.
  \begin{enumerate}
    \item The sequence $(a_k)_{k=0}^\infty$ is eventually periodic if and only if the trajectory $\{T^k(n)\}_{k=0}^\infty$ reaches a cycle.
    \item The sequence $(a_k)_{k=0}^\infty$ is eventually periodic with period $2$ if and only if the trajectory $\{T^k(n)\}_{k=0}^\infty$ reaches the trivial cycle.
    \item The sequence $(a_k)_{k=0}^\infty$ is not eventually periodic if and only if the trajectory $\{T^k(n)\}_{k=0}^\infty$ diverges to infinity.
  \end{enumerate}
\end{proposition}

\begin{proof}
The third assertion follows immediately from the first one. It is also clear that, if the trajectory $\{T^k(n)\}_{k=0}^\infty$ reaches a cycle then the sequence $(a_k)_{k=0}^\infty$ is eventually periodic having period equal to the length of the cycle. So, it remains to prove the inverse implication, which will complete the proof of the first and second assertions.

Assume that $(a_k)_{k=0}^\infty$ is eventually periodic. Therefore, there are $k_0\in \mathbb{N}$ and $l\in \mathbb{N}^*$ (the period of the sequence) such that $a_{k+l}=a_k$ for any $k\ge k_0$. Let $n_0 = T^{k_0}(n)$. Then for any $k\in \mathbb{N}^*$ we have $T^k(n_0) = T^{k+k_0}(n)$ and $(-1)^{T^k(n)} = (-1)^{T^{k+k_0}(n)}$. Consequently, replacing $n$ with $n_0$ if necessary, we may assume without loss of generality that $(a_k)_{k=0}^\infty$ is periodic with period $l\in \mathbb{N}^*$.

Let $d$ denote the number of $-1$'s appearing in any period of the sequence $(a_k)_{k=0}^\infty$ (where $d<l$). By Theorem \ref{th.formula_for_T^k} it follows that
$$T^l(n)= \frac{3^d n +\upsilon}{2^l},$$
where the number $\upsilon$ depends on $l$, $d$ and the positions of $-1$'s in the sequence $(a_1, a_2, \ldots, a_l)$. Similarly, we have:
$$T^{2l}(n) = T^l(T^l(n))= \frac{3^d T^l(n) +\upsilon}{2^l}.$$
Therefore, if we consider the dynamical system:
\begin{align*}
  x_0 & =n \\
  x_{m+1} &= A(x_m) =  \frac{3^d x_m}{2^l}+ \frac{\upsilon}{2^l} =ax_m+b
\end{align*}
then $x_m$ coincides with $T^{ml}(n)$ for every $m\in\mathbb{N}$.

We now distinguish two cases. Assume first that $3^d>2^l$. In this case
$$x_m =A^m(n) = a^m n + \frac{a^m-1}{a-1}\cdot b = a^m \left(n+\frac{b}{a-1} \right) - \frac{b}{a-1}. $$
It follows that:
$$a^m = \frac{x_m (a-1) +b}{n(a-1)+b},$$
or equivalently,
$$\left(\frac{3^d}{2^l}\right)^m=\frac{x_m(3^d-2^l)+\upsilon}{n(3^d-2^l)+\upsilon}.$$
Since $x_m$ is an integer and $2,3$ are relative prime numbers, from the above equation we obtain:
$$2^{ml} \mid n(3^d-2^l)+\upsilon \quad \forall m\in \mathbb{N},$$
and we have reached a contradiction, because $n, d, l, \upsilon$ are fixed and $m$ varies.

Consequently, if the sequence $(-1)^{T^k(n)}$ is eventually periodic, then we must have $3^d<2^l$. Consider again the dynamical system defined by the map $A$. Since $a=\frac{3^d}{2^l}<1$, we get:
$$\lim_{m\to \infty}A^m(n)=-\frac{b}{a-1}.$$
Hence, the sequence $\{A^m(n)\}_{m=0}^\infty$ is bounded in $\mathbb{N}^*$ and so is $\{T^k(n)\}_{k=0}^\infty$. Therefore, the trajectory of $n$ cannot diverge to infinity, and hence it reaches a cycle. This completes the proof of the first assertion.

Finally, in the special case where the sequence $(a_k)_{k=0}^\infty$ has period $2$, we may assume (by omitting $a_0$, if necessary) that the first term is equal to $-1$. Therefore, $l=2$, $d=1$ and by Theorem \ref{th.formula_for_T^k} it follows that $\upsilon=1$. Hence $a=\frac{3}{4}$, $b=\frac{1}{4}$ and it follows that $\lim_{m\to \infty}A^m(n)=-\frac{b}{a-1}=1$, i.e. $T^{2m}(n)=1$ for all sufficiently large $m$. Therefore, the Collatz orbit reaches the trivial cycle and this completes the proof of the second assertion and of the lemma.
\end{proof}

As an immediate consequence of the above proposition, we obtain the next corollaries. The second one provides a reformulation of the Collatz conjecture.

\begin{corollary}\label{cor.periodicity}
  Let $k,n\in\mathbb{N}$. The following are equivalent.
  \begin{enumerate}
    \item $T^k(n) \in \{1,2\}$.
    \item The sequence $\left( (-1)^{T^m(n)} \right)_{m\ge k}$ is periodic with period equal to $2$.
  \end{enumerate}
\end{corollary}

\begin{proof}
  Clearly, the first assertion implies the second one. For the reverse implication, we observe that, by the previous proposition, the periodicity of the sequence $\left( (-1)^{T^m(n)} \right)_{m\ge k}$ implies that the trajectory $\{T^m(n)\}_{m\ge k}$ reaches the trivial cycle, i.e. $T^m_0(n) \in \{1,2\}$ for some $m_0 \ge k$. However, if $m_0>k$, then $T^{m_0-1}(n)=4$, $T^{m_0}(n)=2$, $T^{m_0+1}(n)=1$  and this contradicts the periodicity of $\left( (-1)^{T^m(n)} \right)_{m\ge k}$.
\end{proof}

\begin{corollary}
The following are equivalent:
\begin{enumerate}
  \item The Collatz conjecture is true.
  \item For any positive integer $n$, the sequence $(x_{nk})_{k=1}^\infty$ is eventually periodic with period $2$.
  \item For any positive integer $n$, the sequence $\left( (-1)^{T^k(n)} \right)_{k=1}^\infty$ is eventually periodic with period $2$.
\end{enumerate}
\end{corollary}

In view of the previous result, our purpose now is to transfer the Collatz problem from the ``time domain'' to the ``frequency domain''. In order to proceed, we need some notation. We denote by $\left(\omega_{j,k}\right)_{j=1}^{2^k}$ the $2^k$-th roots of $-1$, that is the roots (in the complex plane) of the cyclotomic equation $z^{2^k}=-1.$ Let us remind that these complex numbers are given by the exponential function via the next formula:
$$\omega_{j,k} = \exp\left(\frac{(2j-1)\pi i}{2^k} \right) \quad \text{for } j=1,2,\ldots, 2^k.$$
For technical reasons we will also make use of the notation $\omega_{m,k}$ for any positive integer $m$, where of course $\omega_{m,k}=\omega_{j,k}$ whenever $m=j \mod 2^k$.

We also need the following lemma describing, for every $k$, the generating function of the sequence $\big((-1)^{T^k(n)}\big)_{n=1}^\infty$ (for the proof we refer to \cite{Levent}, Lemma 2.3).

\begin{lemma}
  For every $k=1,2,\ldots$, the series $g_k(x)=\sum_{n=1}^{\infty} (-1)^{T^k(n)} x^n$ converges (absolutely) for any complex number $x$ with $\abs{x}<1$ and its sum is given by the rational function:
$$g_k(x)=\sum_{n=1}^{\infty} (-1)^{T^k(n)} x^n = \frac{P_k(x)}{1+x^{2^k}},$$
where $P_k$ is the polynomial $P_k(x) = (-1)^{T^k(1)} x + \ldots + (-1)^{T^k(2^k)} x^{2^k}.$
\end{lemma}

Our next result describes $(-1)^{T^k(n)}$ as a linear combination of the numbers $\big(\omega_{j,k}^n\big)_{j=1}^{2^k}$. The sequence of coefficients $\widetilde{p_k}=(b_1,b_2,\ldots, b_n)$, which appear in this theorem, can be seen as the Discrete Fourier Transform of the finite sequence $\left((-1)^{T^k(n)}\right)_{n=1}^{2^k}$.

\begin{theorem}\label{th.formula_for_(-1)^T^k(n)}
For every $k=1,2,\ldots$ and every $n=1,2,\ldots$, the next formula holds:
$$(-1)^{T^k(n)} = \overline{b_1}\omega_{1,k}^n + \overline{b_2} \omega_{2,k}^n + \ldots + \overline{b_{2^k}}\omega_{2^k,k}^n,$$
where the coefficients $\widetilde{p_k}=(b_j)_{j=1}^{2^k}$ are given by:
$$b_j = \frac{P_k(\omega_{j,k})}{2^k}$$
and they satisfy the equation:
$$\sum_{j=1}^{2^k} b_j = \sum_{j=1}^{2^k} \frac{P_k(\omega_{j,k})}{2^k}=1.$$
\end{theorem}

\begin{proof}
To simplify the notation, throughout this proof we fix positive integers $k,n$ and we set $\omega_j=\omega_{j,k}$ for any $j=1,2,\ldots, 2^k$.

By the previous lemma, we know that
$$\sum_{n=1}^{\infty} (-1)^{T^k(n)} x^n = \frac{P_k(x)}{1+x^{2^k}} = \frac{(-1)^{T^k(1)} x + \ldots + (-1)^{T^k(2^k)} x^{2^k}}{1+x^{2^k}}.$$
Since $T^k(2^k)=1$, we can write:
$$\frac{P_k(x)}{1+x^{2^k}} = \frac{(-1)^{T^k(1)} x + \ldots - x^{2^k}}{1+x^{2^k}} = \frac{(-1)^{T^k(1)} x + \ldots +(-1)^{T^k(2^{k-1})} x^{2^k-1}+1}{1+x^{2^k}} -1 = \frac{f_k(x)}{1+x^{2^k}}-1, $$
where $f_k(x)$ is the polynomial:
$$f_k(x) = (-1)^{T^k(1)} x + \ldots +(-1)^{T^k(2^{k-1})} x^{2^k-1}+1 = 1+P_k(x)+x^{2^k}.$$
Therefore,
\begin{equation}\label{eq.coefficients-1}
  \sum_{n=1}^{\infty} (-1)^{T^k(n)} x^n = -1+ \frac{f_k(x)}{1+x^{2^k}}.
\end{equation}
Because of the fact that $(\omega_j)_{j=1}^{2^k}$ are the roots of the polynomial in the denominator, the rational function $\frac{f_k(x)}{1+x^{2^k}}$ can be analysed as follows
\begin{align*}
  \frac{f_k(x)}{1+x^{2^k}} =&  \frac{A_1}{x-\omega_1} + \frac{A_2}{x-\omega_2} + \ldots + \frac{A_{2^k}}{x-\omega_{2^k}} \\
  = & \frac{A_1}{-\omega_1(1-\overline{\omega_1}x)} + \frac{A_2}{-\omega_2(1-\overline{\omega_2}x)} + \ldots + \frac{A_{2^k}}{-\omega_{2^k}(1-\overline{\omega_{2^k}}x)}.
\end{align*}
The complex numbers $(A_j)_{j=1}^{2^k}$ are given by
\begin{align*}
  A_j = & \lim_{x\to \omega_j} -\omega_j \cdot (1-\overline{\omega_j} x)\cdot \frac{f_k(x)}{1+x^{2^k}} =  f_k(\omega_j) \cdot \lim_{x\to \omega_j} \frac{-\omega_j  (1-\overline{\omega_j} x)}{1+x^{2^k}}\\
  = & f_k(\omega_j) \cdot \lim_{x\to \omega_j} \frac{-\omega_j \cdot  (-\overline{\omega_j}) }{2^k\cdot x^{2^k-1}} =  f_k(\omega_j) \cdot \frac{1}{2^k\cdot \omega_j^{2^k-1}} =  \frac{-f_k(\omega_j)\cdot \omega_j}{2^k}.
\end{align*}
By the definition of $f_k$, it follows easily that $f_k(\omega_j)=P_k(\omega_j)$. Hence,
$$A_j = \frac{-P_k(\omega_j)\cdot \omega_j}{2^k}.$$
Substituting in equation \eqref{eq.coefficients-1}, we obtain:
\begin{align*}
  \sum_{n=1}^{\infty} (-1)^{T^k(n)} x^n =  & -1+ \sum_{j=1}^{2^k} \frac{A_j}{-\omega_j(1-\overline{\omega_j}x)} =  -1+ \frac{1}{2^k} \sum_{j=1}^{2^k} \frac{P_k(\omega_j)}{1-\overline{\omega_j}x} \\
  = & -1 + \frac{1}{2^k} \sum_{j=1}^{2^k} \sum_{n=0}^{\infty}P_k(\omega_j)(\overline{\omega_j}x)^n\\
  = & -1 + \frac{1}{2^k}\sum_{n=0}^{\infty}\Big(\sum_{j=1}^{2^k}P_k(\omega_j)(\overline{\omega_j})^n\Big) x^n\\
  = & -1 + \frac{1}{2^k} \sum_{j=1}^{2^k}P_k(\omega_j) +\sum_{n=1}^{\infty}\Big(\frac{1}{2^k}\sum_{j=1}^{2^k}P_k(\omega_j)(\overline{\omega_j})^n\Big) x^n.
\end{align*}
By the above equations we can deduce immediately that:
$$\frac{1}{2^k} \sum_{j=1}^{2^k}P_k(\omega_j)=1$$
and
$$(-1)^{T^k(n)} = \frac{1}{2^k}\sum_{j=1}^{2^k}P_k(\omega_j)\overline{\omega_j}^n.$$
Taking the complex conjugate in the last equation completes the proof.
\end{proof}

The previous theorem provides also a property of the coefficients $\widetilde{p_k}=(b_j)_{j=1}^{2^k}$. Another remark is that the $\ell_2$-norm of this finite sequence of complex numbers is equal to $1$ and it is actually a consequence of Parseval's identity. This is described in the next result.

\begin{theorem}
  For any $k=1,2,\ldots $, if $v_k$ the $\ell_2$-norm of the vector:
$$\widetilde{p_k}=\left(b_1, b_2, \ldots, b_n\right)=\frac{1}{2^k} \left(P_k(\omega_{1,k}), \ldots, P_k(\omega_{2^k,k}) \right),$$
is equal to $1$.
\end{theorem}

\begin{proof}
  To simplify the notation, we fix a positive integer $k$ and we set $\epsilon_j = (-1)^{T^k(j)}$ for every $j=1,2,\ldots, 2^k$, so that for any complex number $x$ we have $P_k(x) = \epsilon_1 x + \epsilon_2 x^2 +\ldots + \epsilon_{2^k} x^{2^k}$.
We calculate the $\ell_2$-norm of the vector $\widetilde{p_k}$:
\begin{align*}
  \|\widetilde{p_k}\|_{\ell_2}^2 = & \frac{1}{2^{2k}} \sum_{j=1}^{2^k} \abs{P_k(\omega_{j,k})}^2 = \frac{1}{2^{2k}} \sum_{j=1}^{2^k} P_k(\omega_{j,k}) \cdot \overline{P_k(\omega_j)} \\
  = & \frac{1}{2^{2k}} \sum_{j=1}^{2^k} \left(\epsilon_1 \omega_{j,k} + \epsilon_2 \omega_{j,k}^2 +\ldots + \epsilon_{2^k} \omega_{j,k}^{2^k} \right) \cdot \left(\epsilon_1 \overline{\omega_{j,k}} + \epsilon_2 \overline{\omega_{j,k}^2} +\ldots + \epsilon_{2^k} \overline{\omega_{j,k}^{2^k}} \right) \\
  = & \frac{1}{2^{2k}}
\left[\begin{array}{cccc}
  \epsilon_1 &\epsilon_2 & \ldots & \epsilon_{2^k}
\end{array}\right]
W_k W_k^*
\left[\begin{array}{cccc}
  \epsilon_1 &\epsilon_2 & \ldots & \epsilon_{2^k}
\end{array}\right]^t,
\end{align*}
where $W_k$ is the $2^k\times 2^k$ matrix
$$ W_k= \left[ \omega_{j,k}^m \right]_{j,m=1}^{2^k} = \left[\begin{array}{cccc}
  \omega_1 & \omega_2 & \ldots & \omega_{2^k} \\
  \omega_1^2 & \omega_2^2 & \ldots & \omega_{2^k}^2 \\
  \vdots & \vdots & \vdots & \vdots \\
  \omega_1^{2^k} & \omega_2^{2^k} & \ldots & \omega_{2^k}^{2^k}
\end{array}\right],$$
and $W_k^*$ is the transpose of $W_k$. By the elementary properties of the cyclotomic roots, it is easy to observe that
$$W_k W_k^* = \left[ \sum_{j=1}^{2^k} \omega_j^r \cdot \overline{\omega_j^m} \right]_{r,m=1}^{2^k} =2^k I_{2^k},$$
where $I_{2^k}$ is the $2^k \times 2^k$ identity matrix. Hence, the desired result can now be deduced after some routine calculations.
\end{proof}

Theorem \ref{th.formula_for_(-1)^T^k(n)} describes the sign $(-1)^{T^k(n)}$ as a linear combination of the $n$-th power of the $2^k$ roots of $-1$. However, since having period $2$ means that $(-1)^{T^{k+2}(n)} = (-1)^{T^k(n)}$, it would be useful to be able to connect $(-1)^{T^{k+1}(n)}$ directly with $(-1)^{T^k(n)}$. The main step towards this direction is the theorem that follows.

\begin{theorem}\label{th.roots of -1}
For any $k$ and any $j=1,2, \ldots, 2^k$ the following equation holds:
$$\omega_{j,k}^{T(n)} = \frac{1}{2} \cdot \omega_{j,k+1}^n +\frac{\omega_{j,k+1}}{2} \cdot \omega_{3j-1,k+1}^n +\frac{1}{2} \omega_{2^k+j,k+1}^n - \frac{\omega_{j,k+1}}{2} \cdot \omega_{2^k+3j-1,k+1}^n.$$
\end{theorem}

\begin{proof}
Once again, we wish to simplify the notation. For this reason, we fix a positive integer $k$ and throughout this proof we set $\omega_j=\omega_{j,k}$ for any $j=1,2,\ldots,2^k$ and $z_j=\omega_{j,k+1}$ for any $j=1,2,\ldots, 2^{k+1}$. Thus, the equation we want to prove takes the form:
$$\omega_j^{T(n)} = \frac{1}{2} \cdot z_j^n +\frac{z_j}{2} \cdot z_{3j-1}^n +\frac{1}{2} z_{2^k+j}^n - \frac{z_j}{2} \cdot z_{2^k+3j-1}^n.$$

Following one of our main techniques, for an arbitrary (but fixed) $j\in\{1,2,\ldots,2^k\}$, we consider the powerseries $\sum_{n=1}^{\infty} \omega_j^{T(n)} x^n$, which clearly converges for any complex number $x$ with $\abs{x}<1$. The sum of this series is written as follows:
\begin{align*}
  \sum_{n=1}^{\infty} \omega_j^{T(n)} x^n = & \sum_{m=1}^{\infty} \omega_j^{T(2m)} x^{2m} + \sum_{m=0}^{\infty} \omega_j^{T(2m+1)} x^{2m+1}\\
  = & \sum_{m=1}^{\infty} \omega_j^{m} x^{2m} + \sum_{m=0}^{\infty} \omega_j^{3m+2} x^{2m+1} \\
  = & \sum_{m=1}^{\infty} (\omega_j \cdot x^2)^m + \omega_j^2 \cdot x \cdot \sum_{m=0}^{\infty} (\omega_j^3 \cdot x^2)^m \\
  = & \frac{\omega_j \cdot x^2}{1-\omega_j \cdot x^2} + \frac{\omega_j^2 \cdot x}{1-\omega_j^3 \cdot x^2}\\
  = & \frac{-1+\omega_j \cdot x^2}{1-\omega_j \cdot x^2} + \frac{1}{1-\omega_j \cdot x^2} + \frac{\omega_j^2 \cdot x}{1-\omega_j^3 \cdot x^2}\\
  = & -1 + \frac{1}{1-\omega_j \cdot x^2} + \frac{\omega_j^2 \cdot x}{1-\omega_j^3 \cdot x^2}\\
\end{align*}
Since $z_j^2=\omega_j$, using partial fraction decomposition, we may write
$$\frac{1}{1-\omega_j \cdot x^2} = \frac{1}{(1-z_j \cdot x)(1+z_j \cdot x)} = \frac{1}{2(1-z_j x)}+\frac{1}{2(1+z_j x)}.$$
Similarly, $z_{3j-1}^2 = \omega_j^3$, and a simple partial fraction decomposition shows that:
$$\frac{\omega_j^2 \cdot x}{1-\omega_j^3 \cdot x^2}= \frac{\omega_j^2 \cdot x}{(1-z_{3j-1}x)(1+z_{3j-1}x)} = \frac{z_1^{2j-1}}{2(1-z_{3j-1}x)}+\frac{z_1^{2j-1}}{2(1+z_{3j-1}x)}.$$
Consequently, the sum of the powerseries becomes
\begin{align*}
  \sum_{n=1}^{\infty} \omega_j^{T(n)} x^n = & -1 + \frac{1}{2(1-z_j x)} + \frac{1}{2(1+z_j x)} + \frac{z_1^{2j-1}}{2(1-z_{3j-1}x)} +\frac{z_1^{2j-1}}{2(1+z_{3j-1}x)} \\
  = & -1 +\frac{1}{2} \sum_{n=0}^{\infty} (z_j\cdot x)^n + \frac{1}{2} \sum_{n=0}^{\infty} (-z_j\cdot x)^n + \frac{z_1^{2j-1}}{2} \sum_{n=0}^{\infty} (z_{3j-1}\cdot x)^n\\
    & - \frac{z_1^{2j-1}}{2} \sum_{n=0}^{\infty} (-z_{3j-1}\cdot x)^n \\
  = & \sum_{n=1}^{\infty} \left[\frac{1}{2} \cdot z_j^n + \frac{1}{2} \cdot (-z_j)^n + \frac{z_1^{2j-1}}{2} \cdot z_{3j-1}^n -  \frac{z_1^{2j-1}}{2} \cdot (-z_{3j-1})^n \right] \cdot x^n.
\end{align*}
The above equation implies that:
$$ \omega_j^{T(n)} =\frac{1}{2} \cdot z_j^n + \frac{1}{2} \cdot (-z_j)^n + \frac{z_1^{2j-1}}{2} \cdot z_{3j-1}^n - \frac{z_1^{2j-1}}{2} \cdot (-z_{3j-1})^n.$$
By the exponential form of the roots $(z_j)_{j=1}^{2^{k+1}}$ of $-1$, it follows immediately that:
$$-z_j=z_{2^k+j}, \quad -z_{3j-1} = z_{2^k+3j-1} \quad \text{and} \quad z_i^{2j-1}=z_j.$$
Hence,
$$\omega_j^{T(n)} = \frac{1}{2} \cdot z_j^n +\frac{z_j}{2} \cdot z_{3j-1}^n +\frac{1}{2} z_{2^k+j}^n - \frac{z_j}{2} \cdot z_{2^k+3j-1}^n,$$
and the result has been proved.
\end{proof}

\begin{example}
  For $k=0$ we have $\omega_{1,0}=-1$ and $\omega_{1,1}=i$, $\omega_{2,1}=-i$. The formula of Theorem \ref{th.roots of -1} gives the equation
$$(-1)^{T(n)} = \frac{1-i}{2} i^n + \frac{1+i}{2}(-i)^n,$$
which is in complete agrement with Theorem \ref{th.formula_for_(-1)^T^k(n)}.
For $k=1$, we obtain:
$$i^{T(n)} = \frac{1}{2} \omega_1^n+\frac{\omega_1}{2}\omega_2^n+\frac{1}{2} \omega_3^n-\frac{\omega_1}{2}\omega_4^n$$
and
$$(-i)^{T(n)} = \frac{\omega_2}{2} \omega_1^n +\frac{1}{2} \omega_2^n-\frac{\omega_2}{2}\omega_3^n+\frac{1}{2}\omega_4^n.$$
Hence, we have:
$$
  (-1)^{T^2(n)} = (-1)^{T(T(n))}=  \frac{1-i}{2} i^{T(n)} + \frac{1+i}{2}(-i)^{T(n)}.
$$
\end{example}

\section{An isometry on a Hilbert space}\label{sec.ell-2-oper}

Theorem \ref{th.roots of -1} describes the quantity $\omega_{j,k}^{T(n)}$, for any $2^k$-th root $\omega_{j,k}$ of $-1$, as a linear combination of $(\omega_{j, k+1}^n)_{j=1}^{2^{k+1}}$, where the coefficients are independent from $n$. Motivated by this remark, we now consider the functions $\Omega_{j,k} \colon \mathbb{N} \to \mathbb{C}$, for any $j\in\mathbb{N}^*$ and $k\in\mathbb{N}$, such that
$$\Omega_{j,k} (n) = \omega_{j,k}^n \quad \forall n\in\mathbb{N}.$$
Recall that $\Omega_{m,k}=\Omega_{j,k}$ for any $m = j \mod 2^k$. Furthermore, let $V_k\subset \mathbb{C}^\mathbb{N}$ be the vector space generated by $(\Omega_{j, k}^n)_{j=1}^{2^{k}}$, i.e.
$$V_k = \left\langle\Omega_{1, k}, \Omega_{2, k}, \ldots, \Omega_{2^k, k} \right\rangle= \Big\{ \sum_{j=1}^{2^k} c_j \cdot \Omega_{j,k} \mid c_j\in\mathbb{C}, \, j=1,2,\ldots,2^k\Big\}.$$
This space is equipped with inner product, i.e.
$$\sca{\sum_{j=1}^{2^k} c_j \cdot \Omega_{j,k}, \sum_{j=1}^{2^k} d_j \cdot \Omega_{j,k}} = \sum_{j=1}^{2^k} c_j \overline{d_j},$$
for any $c_j,d_j\in\mathbb{C}$, $j=1,2,\ldots,2^k$, and the corresponding $\ell_2$-norm:
$$\Big\|\sum_{j=1}^{2^k} c_j \cdot \Omega_{j,k} \Big\| = \Big(\sum_{j=1}^{2^k} \abs{c_j}^2 \Big)^{1/2},$$
so that $V_k$ is isometrically isomorphic to $\ell_2^{2^k}$ via the isometry mapping $\Omega_{j,k}$ to $e_j$ for $j=1,2,\ldots,2^k$.

For every natural number $k$, the $3x+1$-function $T\colon \mathbb{N}\to \mathbb{N}$ defines (through Theorem \ref{th.roots of -1}) a linear operator
$$A_k \colon V_k \to V_{k+1}$$
whose values on the vector basis of $V_k$ are given by the formula:
\begin{equation}\label{eq.formula_of_A_k}
  A_k(\Omega_{j, k}) = \frac{1}{2} \cdot \Omega_{j,k+1} +\frac{\omega_{j,k+1}}{2} \cdot \Omega_{3j-1,k+1} +\frac{1}{2} \Omega_{2^k+j,k+1} - \frac{\omega_{j,k+1}}{2} \cdot \Omega_{2^k+3j-1,k+1}.
\end{equation}
As usual, the linear operator $A_k$ can be described with a $2^k\times 2^{k+1}$ matrix, which we denote by $M_k$. The above formula gives the $i$-th line of the matrix $M_k$. Therefore, for every complex numbers $(c_j)_{j=1}^{2^k}$, we have:
$$A_k(c_1\Omega_{1,k} + \ldots + c_{2^k} \Omega_{2^k,k}) = M_k^t \cdot
\left[\begin{array}{c}
        c_1 \\
        c_2 \\
        \vdots \\
        c_{2^k}
      \end{array}
\right].$$

\begin{example}\label{example-matrices}
Below we write down the matrix $M_k$ for the first values of $k$:
$$M_1= \frac{1}{2} \left[ \begin{array}{cc}
         1-i & 1+i
       \end{array} \right]
\quad \quad
M_2 = \frac{1}{2} \left[\begin{array}{cccc}
          1 & \omega_{1,2} & 1 & -\omega_{1,2} \\
          \omega_{2,2} & 1 & -\omega_{2,2} & 1
        \end{array} \right]$$
$$M_3= \frac{1}{2} \left[\begin{array}{cccccccc}
               1 & \omega_{1,3} & 0 & 0 & 1 & -\omega_{1,3} & 0 & 0 \\
               -\omega_{2,3} & 1 & 0 & 0 & \omega_{2,3} & 1 & 0 & 0 \\
               0 & 0 & 1 & -\omega_{3,3} & 0 & 0 & 1 & \omega_{3,3} \\
               0 & 0 & \omega_{4,3} & 1 & 0 & 0 & -\omega_{4,3} & 1
             \end{array} \right]$$
\end{example}

\begin{theorem}
For any positive integer $k$ the following equations hold:
  \begin{enumerate}
    \item $(-1)^{T^k(n)}= A_{k-1}\left( (-1)^{T^{k-1}(\cdot)} \right)(n)$ for any $n\in\mathbb{N}^*$;
    \item $(-1)^{T^k(n)} = A_{k-1}A_{k-2} \ldots A_1 A_0 \left((-1)^{T^0(\cdot)} \right)(n)$ for any $n\in\mathbb{N}^*$.
  \end{enumerate}
\end{theorem}

\begin{proof}
  The second assertion is an immediate consequence of the first one. The first assertion is proved by induction on $k$. First, we observe that, by the definition of $A_k$, we have
$$A_k(\Omega_{j,k})(n) = \omega_{j,k}^{T(n)}.$$
Furthermore, by Theorem \ref{th.formula_for_(-1)^T^k(n)}, we know that:
$$(-1)^{T^k(n)} = \sum_{j=1}^{2^k} \overline{b_j} \omega_{j,k}^n .$$
Replacing $n$ with $T(n)$ in the above equation implies
$$(-1)^{T^{k+1}(n)} = \sum_{j=1}^{2^k} \overline{b_j} \omega_{j,k}^{T(n)} = \sum_{j=1}^{2^k} \overline{b_j} A_k\left(\Omega_{j,k}\right)(n) = A_k\left( \sum_{j=1}^{2^k} \overline{b_j} \Omega_{j,k}\right)(n) = A_k\left( (-1)^{T^k(\cdot)} \right)(n).$$
\end{proof}

\subsection*{Structure and properties of the operator $A_k$ and the matrix $M_k$}
In this subsection, we delve deeper into the properties of the operators $(A_k)_{k=0}^\infty$ and the matrices $(M_k)_{k=0}^\infty$ defined above. The aforementioned examples are quite enlightening.

It follows, by equation \eqref{eq.formula_of_A_k}, that for each $j=1,2,\ldots,2^k$ the $j$-th line of the matrix $M_k$ has exactly $4$ non zero entries. Two of them are located at the $j$ and $2^k+j$ columns, i.e. at the $(j,j)$ and $(j,2^k+j)$ entries of $M_k$, and they are equal to $\frac{1}{2}$. Let us now divide $M_k$ into two $2^k\times 2^k$ submatrices: $M_k^L$ containing the first $2^k$ columns of $M_k$ and $M_k^R$ containing the last $2^k$ columns of $M_k$. Then the diagonal entries of $M_k^L$ and $M_k^R$ are all equal to $\frac{1}{2}$. Thus, we may write $M_k^L=\frac{1}{2} I-L_k$ and $M_k^R = \frac{1}{2} I-R_k$, where $I$ is the $2^k\times 2^k$ identity matrix and $L_k$, $R_k$ are $2^k\times 2^k$ matrices with diagonal entries equal to $0$.

By equation \eqref{eq.formula_of_A_k}, we also observe that the third non zero element of the $j$-th line of $M_k$ lies at the $3j-1$ column. However, $3j-1$ may be greater that $2^{k+1}$. Therefore, it is more accurate to say that this element lies at the $(3j-1)\mod 2^{k+1}$ column where, however $0 \mod 2^{k+1}$ refers to the $2^{k+1}$-th column of the matrix. Similarly, the fourth non zero element of the $j$-th line lies at the $(2^k+3j-1)\mod 2^{k+1}$ column. These elements are opposite to each other: one is equal to $\frac{\omega_{j,k+1}}{2}$ and the other one is $-\frac{\omega_{j,k+1}}{2}$. We also observe that the ``distance" between their positions in the $j$-th line of $M_k$ is always equal to $2^k$. This implies that one of these elements lies at the matrix $M_k^L$, and hence at the matrix $L_k$, and the other one lies at the corresponding position of the matrix $R_k$. Consequently, we have $R_k=-L_k$.

Therefore, we have the next proposition.

\begin{proposition}\label{prop.properties_of_M}
 For every positive integer $k$, the matrix $M_k$ has the following properties.
\begin{enumerate}
  \item $M_k$ consists of two submatrices $M_k^L$ and $M_k^R$ whose diagonal entries are all equal to $\frac{1}{2}$.
  \item If we write $M_k^L=\frac{1}{2}I-L_k$ and $M_k^R=\frac{1}{2}I-R_k$, then $R_k=-L_k$.
  \item For each $j=1,2,\ldots, 2^k$, the $j$-line of the matrix $L_k$ (and of $R_k$) contains exactly one non zero element, which lies at the $(3j-1)\mod 2^k$ column.
  \item Each column of the matrix $L_k$ (and of $R_k$) has exactly one non zero element.
\end{enumerate}
\end{proposition}

\begin{proof}
  The first two assertions follow immediately by the discussion preceding the proposition. For the third assertion, we know that the non zero element of the $j$-line of $L_k$ lies either at the $(3j-1) \mod 2^{k+1}$ or at the $(2^k+3j-1) \mod2^{k+1}$ column according to which of the two numbers is between $1$ and $2^k$. However, since $(3j-1) \mod 2^{k+1}=(3j-1) \mod 2^k$ and $(2^k+3j-1) \mod2^{k+1} = (3j-1) \mod 2^k$, it can be deduced that the non zero element of the $j$-line is located at the $(3j-1) \mod 2^k$ column.

  Finally, for the fourth assertion, it suffices to observe that for any $1\le j <\lambda \le 2^k$ we have $3\lambda-1 \ne 3j-1 \mod 2^k$ and, hence, each column of $L_k$ contains exactly one non zero element.
\end{proof}

\begin{remark}
  It is perhaps worth noting that the number $3$, which actually determines the position of the non zero entries in $M_k$ not lying at the diagonal of $M_k^L$ or $M_k^R$, comes directly form the Collatz function. Indeed, in the proof of Theorem \ref{th.roots of -1} we split out the sum $\sum_{n=1}^{\infty} \omega_j^{T(n)} x^n$ into two parts: one containing the even power of $x$ and the other containing the odd powers of $x$. The proof of the aforementioned theorem shows that the sum of the even powers of $x$ produces the diagonal elements of $M_k^L$ and $M_k^R$. The rest non zero elements are produced by the sum of the even powers of $x$. Consequently, the number $3$ from the Collatz function make his presence in the matrix $M_k$ by showing the position of the non zero entries not belonging to the diagonal of $M_k^L$ or $M_k^R$.
\end{remark}

Let now $L_{k,P}$ be the matrix which results from $L_k$ by replacing the non zero elements of the latter with $1$. By the previous proposition, it follows that $L_{k,P}$ is a permutation matrix.

\begin{proposition}\label{prop.permutation}
  For any $k\ge 2$, the permutation $\sigma_k$ of the symmetric group $S_{2^k}$ corresponding to $L_{k,P}$ is the product of two disjoint cycles, each one having length equal to $2^{k-1}$.
\end{proposition}

In order to prove the proposition, we need the next lemma, which can be proved by induction and elementary number theory.

\begin{lemma}
  The following hold:
  \begin{enumerate}
    \item Let $k\ge 2$ be any integer. Then, $2^{k-1}$ is the least positive integer $j$ satisfying the property $3^j = 1\mod 2^{k+1}$.
    \item For any positive integer $j$, we have that $3^j-5$ is not a multiple of $8$.
  \end{enumerate}
\end{lemma}

\begin{proof}[Proof of Proposition \ref{prop.permutation}]
By Proposition \ref{prop.properties_of_M}, it follows that $\sigma_k$ is the permutation assigning to each $j\in \{1,2,\ldots, 2^k\}$ the number $3j-1\mod 2^k \in \{1,2,\ldots, 2^k\}$. We claim that this permutation is the product of two cycles: the cycle containing $1$ and the cycle containing $3$.

Let $a_1=1$ and let $\sigma_k^1=(a_1, a_2, a_3, \ldots )$ be the cycle starting from $1$. It is easy to see that for every $j\ge 1$, one has
$$a_j = (3^{j-1} - 3^{j-2}-\ldots-3-1) \mod 2^k = \frac{3^{j-1}+1}{2} \mod 2^k.$$
By the first assertion of the previous lemma, we obtain that
$$a_{2^{k-1}+1} = 1 \mod 2^k, \quad a_j \ne 1\mod 2^k \,\, \forall j=2,3,\ldots, 2^{k-1}.$$

Similarly, let $\sigma_k^3=(b_1, b_2, b_3, \ldots )$ be the cycle starting from $b_1=3$. Then,
$$b_j = (3^j - 3^{j-2} - \ldots -3-1)\mod2^k = \frac{5\cdot 3^{j-1}+1}{2} \mod 2^k.$$
Using again the first assertion of the previous lemma, we conclude that
$$b_{2^{k-1}+1} = 3 \mod 2^k, \quad  b_j \ne 3\mod 2^k \,\, \forall j=2,3,\ldots, 2^{k-1}.$$

Consequently, the two cycles have length $2^{k-1}$. Finally, the number $3$ does not belong to the first cycle, since if this was the case, then for some $j\ge 1$ we would have
$$ \frac{3^{j-1}+1}{2} = 3 \mod 2^{k}.$$
The above equation, however, implies that $3^{j-1} -5 =0\mod 2^{k+1}$, and according to the previous lemma we have a contradiction. Therefore, the cycles $\sigma_k^1, \sigma_k^3$ are disjoint.
\end{proof}

Since, for every $m\in\mathbb{N}$, the non zero elements of $L_k^m$ and $L_{k,P}^m$ are located at the same entries of the two matrices, we obtain the next result.

\begin{corollary}
  For each $k$, $L_k^{2^{k-1}}$ is a diagonal matrix.
\end{corollary}

Finally, the fundamental result for the matrices $M_k$ are given in the next theorem.

\begin{theorem}
For any $k\in \mathbb{N}$ we have $M_k M_k^*=I_{2^k}$, where $I_{2^k}$ is the $2^k \times 2^k$ matrix.
\end{theorem}

\begin{proof}
For $k=0,1,2$, it is easy to verify that $M_k M_k^* =I_{2^k}$ (see the matrices in Example \ref{example-matrices}). Hence, we may assume that $k> 2$.

For any $k\ge 2$ and any $j=1,2,\ldots, 2^k$, the $j$-th diagonal element of the matrix $M_k M_k^*$ is given by:
$$\frac{1}{4} \left( 1+ \abs{\omega_{j,k+1}}^2 + 1 +\abs{\omega_{j,k+1}}^2 \right) =1.$$
Therefore, we have to verify that the non-diagonal elements of the above matrix are all zero. To this end, we fix integers $1\le j,\lambda \le 2^k$, with $j\ne \lambda$, and we prove that the product if the $j$-th line of $M_k$ with the $\lambda$-th column of $M_k^*$ is zero.

Firstly, we observe that the $j$-th line of $M_k$ comprises the elements $\frac{1}{2}$, $\frac{\omega_{j,k+1}}{2}$, $\frac{1}{2}$ and $-\frac{\omega_{j,k+1}}{2}$ which are located at the columns $j$, $(3j-1) \mod 2^{k+1}$, $2^k+j$ and $2^k+3j-1 \mod 2^{k+1}$ respectively. Clearly, an analogous statement holds for the $\lambda$-th column of the matrix $M_k^*$.

We have assumed that $j\ne\lambda$. Moreover, since $1\le j,\lambda, \le 2^k$, it follows that $2^k+j\ne \lambda$. Hence, we have to distinguish the following cases for the numbers $j,\lambda$, i.e. the place of the non zero elements of the $j$-th line and $\lambda$-th column of $M_k$, $M_k^*$ respectively.

\begin{enumerate}
  \item Suppose that $j=(3\lambda-1) \mod 2^{k+1}$. In this case, we clearly have $2^k+j = (2^k+3\lambda-1) \mod 2^{k+1}$. If we assume that $\lambda= 3j-1 \mod 2^{k+1}$, then we obtain:
      $$3j = (9\lambda-3) \mod 2^{k+1} \quad \text{and} \quad 3j= \lambda+1 \mod 2^{k+1}.$$
      Hence,
      $$9\lambda-3 = \lambda+1 \mod 2^{k+1},$$
      or, equivalently, $8\lambda-4= 0 \mod 2^{k+1}$. Consequently, $4\lambda-1=0\mod 2^{k-1}$, and, since $k> 2$, it follows that $4\lambda-1$ is even, which is a contradiction.

      In a similar way, if we assume that $2^k+\lambda = 3j-1 \mod 2^{k+1}$, then (using the fact that $k>2$) we reach a contradiction. Therefore, the $(j,\lambda)$-entry of the matrix $M_kM_k^*$ is equal to:
      $$\frac{\overline{\omega}_{\lambda,k+1}}{4} - \frac{\overline{\omega}_{\lambda,k+1}}{4} =0.$$

  \item Assume that $j=(3\lambda-1+2^k) \mod 2^{k+1}$. Then, $2^k+j=(3\lambda-1) \mod 2^{k+1}$. As in the previous case, it follows with similar arguments that $\lambda\ne (3j-1) \mod 2^{k+1}$ and $2^k+\lambda \ne (3j-1) \mod 2^{k+1}$. Consequently, the $(j,\lambda)$-entry of the matrix $M_kM_k^*$ is equal to:
      $$\frac{\overline{\omega}_{\lambda,k+1}}{4} - \frac{\overline{\omega}_{\lambda,k+1}}{4} =0.$$

  \item Finally, we assume that $j\ne (3\lambda-1) \mod 2^{k+1}$ and $j \ne (3\lambda-1+2^k) \mod 2^{k+1}$. Then, in the case where either $(3j-1) \mod 2^{k+1} =\lambda$ or $(3j-1) \mod 2^{k+1} = 2^k+\lambda$, it follows that the $(j,\lambda)$-entry of the matrix $M_kM_k^*$ is equal to:
      $$\frac{\omega_{j,k+1}}{4} - \frac{\omega_{j,k+1}}{4} =0.$$
      In all the other cases, the $(j,\lambda)$-entry of the matrix $M_kM_k^*$ is trivially equal to $0$ and the proof is complete.
\end{enumerate}
\end{proof}

The next corollary is straightforward.

\begin{corollary}\label{cor-A_k-isometry}
For any $k\in \mathbb{N}$ we have $A_k^* \circ A_k =id_{V_k}$ and thus $A_k$ is an isometric embedding.
\end{corollary}

\begin{remark}
Since $A_k$ cannot be onto $V_{k+1}$ the equation $A_k \circ A_k^* =id$ is not valid.
\end{remark}

\subsection*{An isometry on a Hilbert space}
We are now ready to ``stick together" the operators $(A_k)_{k=1}^\infty$ and produce an operator on a Hilbert space. More precisely, we consider the direct sum of the spaces $(V_k)_{k=0}^\infty$, i.e.
$$V_\infty = \bigoplus_{k=0}^\infty V_k =  \left\{ x=(x_k)_{k=0}^\infty \mid x_k\in V_k \text{ for all } k\in \mathbb{N} \text{ and } \sum_{k=0}^{\infty} \|x_k\|^2<\infty \right\}.$$
The space $V\infty$ is equipped with inner product: for any $x=(x_k),y=(y_k)\in V_\infty$ we have
$$\sca{x,y} = \sum_{k=0}^{\infty} \sca{x_k, y_k}.$$
Hence, for any $x\in V_\infty$, its norm is given by $\|x\|= \left(\sum_{k=0}^{\infty} \|x_k\|^2\right)^{1/2}$. It is not hard to see that $V_\infty$ is a Hilbert space and in particular it is isometrically isomorphic to the space $\ell_2$ of all square summable sequences.

We next define the operator $S$ as the direct sum of the operators $(A_k)_{k=1}^\infty$. That is,
$$S \colon  V_\infty  \longrightarrow V_\infty$$
with
$$S \left((x_k)_{k=0}^\infty \right) = \left(0,A_0 x_0, A_1x_1, A_2x_2, \ldots \right).$$
Hence, $S$ is a ``shift-like'' operator. The next theorem summarizes the basic properties of $S$.

\begin{theorem}
  The following statements hold:
\begin{enumerate}
  \item The adjoint $S^*\colon V_\infty \to V_\infty$ of the operator $S$ is given by:
  $$S^* (x) = S^*((x_k)_{k\in\mathbb{N}} = \left(A_0^* (x_1), A_1^* (x_2), A_2^* (x_3), \ldots \right) \quad \forall x=(x_k)_{k\in\mathbb{N}}\in V_\infty.$$
  \item The operator $S$ is left unitary, that is $S^* S = id_{V_\infty}$.
  \item The operator $S$ is an isometric embedding.
\end{enumerate}
\end{theorem}

\begin{proof}
The first assertion is straightforward, since,
$$\sca{x,S^*(y)} = \sum_{k=0}^{\infty} \sca{x_k, A_k^*(y_{k+1})} = \sum_{k=0}^{\infty} \sca{A_k(x_k), y_{k+1}} = \sca{S(x), y},$$
for all $x,y\in V_\infty$. The second assertion is an immediate consequence of the first one and Corollary \ref{cor-A_k-isometry}. Finally, the third assertion is clear.
\end{proof}

Bounded linear operators on $\ell_2$ have a natural representation by an infinite matrix. If $L\colon \ell_2 \to \ell_2$ is a bounded linear operator, then the matrix $A=(a_{i,j})$ corresponding to $L$ is defined by the relation $L(e_i) = \sum_{j=1}^{\infty} a_{i,j} e_j$. By the definition of $S$ it is easy to obtain its matrix representation.

\begin{proposition}
  The matrix $M$ corresponding to the isometry $S \colon V_\infty \to V_\infty$ is given by
$$M= \left[
\begin{array}{ccccc}
  0 & M_0 & 0_{1,4} & 0_{1,8} & \ldots  \\
  0 & 0_{2,2} & M_1 & 0_{2,8} & \ldots  \\
  0 & 0_{4,2} & 0_{4,4} & M_2 & \ldots  \\
  \vdots & \vdots & \vdots & \vdots &
\end{array}
\right],$$
where $0_{i,j}$ denotes the $(i,j)$-zero matrix.
\end{proposition}

\section{Wold decomposition of the isometry $S$}\label{sec.Wold-decomp}
The operator $S$ defined in the previous sections is an isometry on the Hilbert space $V_\infty (\approx \ell_2)$. Therefore, the Wold-von Neummann decomposition can be applied (see for example \cite{Halmos}). This decomposition asserts that every isometry $S$ on a Hilbert space can be written as a direct sum $S=(\oplus_{a\in A}S_a) \oplus U$, where $U$ is a unitary operator (possibly vacuous), $S_a$ is the unilateral shift on a Hilbert space $H_a$ and the spaces $(H_a)_{a\in A}$ are isomorphic to each other. This section's purpose is to describe the above decomposition of the operator $S$. We start with the ``unitary part'' $U$. Since, $S$ itself is a shift-like operator, the next proposition is not surprising.

\begin{proposition}
The unitary part $U$ of the operator $S$ is zero.
\end{proposition}

\begin{proof}
In order to find the unitary part of the isometry, one has to apply the operator $S$ successively and take the space $V_u = \bigcap_{n=0}^\infty S^n(V_\infty)$. This space is $S$-invariant and the unitary part is the restriction of $S$ to this subspace. However, $S$ is a shift-like operator, in the sense that $S^n(V_\infty) \subseteq \cup_{j=n}^\infty A_j$. Consequently, $V_u$ is the trivial space and the unitary part of $S$ is zero.
\end{proof}

\begin{remark}
On account of the fact that we are interested mainly in the limit behaviour of $S$ under iteration, it is rather unfortunate that the unitary part of $S$ does not exist. Under different circumstances (probably using some other operator), a non-zero unitary part could be more interesting for the Collatz conjecture.
\end{remark}

The above proposition implies that the operator $S$ can be written as a sum $S=\oplus_{a\in A}S_a$, where $S_a$ is the unilateral shift on a Hilbert space $H_a$. Recall from Section \ref{sec.ell-2-oper} that the space $V_\infty$ is the direct sum of the finite dimensional spaces $\{V_k\}_{k=0}^\infty$ and the operator $S$ is the direct sum of the operators $A_k \colon V_k \to V_{k+1}$, thus we have:
$$V_0 \stackrel{S}{\longrightarrow} V_1 \stackrel{S}{\longrightarrow} V_2 \stackrel{S}{\longrightarrow} V_3 \ldots$$
Let $e_{10}$ be any vector of $V_0$ with norm $1$. We consider the sequence of vectors $e_{10}, S(e_{10}), S^2(e_{10}), \ldots$, where $S^i(e_{10}) \in V_i$ for every $i\in \mathbb{N}$, and we set:
$$H_0 = \overline{\textrm{span}} \{ S^i(e_{10}) \mid i=0,1,2,\ldots\} \subset V_\infty.$$
It is not hard to see that the subspace $H_0$ of $V_\infty$ is isometrically isomorphic to the Hilbert space $\ell_2$. Furthermore, $H_0$ is $S$-invariant and the restriction $S_0$ of $K$ to $H_0$ is given by
$$ S_0\left( \sum_{i=0}^{\infty} \lambda_i S^i(e_{10}) \right)= S\left( \sum_{i=0}^{\infty} \lambda_i S^i(e_{10}) \right) = \sum_{i=0}^{\infty} \lambda_i S^{i+1}(e_{10}),$$
i.e. $S_0$ is the unilateral shift of $H_0$.

We next observe that $S(V_0)=\sca{S(e_{10})}$ is a one-dimensional subspace of $V_1$. Let us denote by $M_1$ the orthogonal complement of $S(V_0)$ in $V_1$. Then $M_1$ is also one-dimensional. We repeat the previous procedure by considering any norm-one vector $e_{11}$ in $M_1$ and setting $H_1 = \overline{\textrm{span}} \{ S^i(e_{11}) \mid i=0,1,2,\ldots\} \subset V_\infty$. Then, $H_1$ is a $S$-invariant subspace of $V_\infty$ and the restriction $S_1$ of $S$ into $H_1$ is the unilateral shift.

In the general case, we consider the orthogonal complement $M_k$ of $S(V_{k-1})$ in $V_k$. Then $M_k$ is a $2^{k-1}$-dimensional subspace of $V_k$. We fix an orthonormal basis $\{e_{jk}\}_{j=1}^{2^{k-1}}$ of $M_k$. For each $j=1,2,\ldots, 2^{k-1}$, we set $H_{jk} = \overline{\textrm{span}} \{ S^i(e_{jk}) \mid i=0,1,2,\ldots\} \subset V_\infty$. Then, $\{H_{jk}\}_{j=1}^{2^{k-1}}$ are isomorphic to $\ell_2$, $S$-invariant subspaces of $V_\infty$ and the restriction $S_{jk}$ of $S$ into $H_{jk}$ is the unilateral shift.

By the above construction, we get immediately the decomposition of the operator $S$.

\begin{proposition}
  The following hold.
  \begin{enumerate}
    \item The spaces $\{H_{10}, H_{jk} \mid k=1,2,\ldots, j=1,2,\ldots, 2^{k-1} \}$ are isometrically isomorphic to $\ell_2$, $S$-invariant subspaces of $V_\infty$ and the restriction $S_{jk}$ of $S$ to $H_{jk}$ is the unilateral shift.
    \item The space $V_\infty$ is the direct sum of $H_{jk}$.
    \item The operator $S$ is the sum of the copies $S_{jk}$ of the unilateral shift.
  \end{enumerate}
\end{proposition}

\section{Collatz conjecture and $C^*$-algebras} \label{sec.C^star-alg}
In this section, we utilize the $C^*$-algebra generated by the isometry $S$ (defined in Section \ref{sec.ell-2-oper}) as well as the $C^*$-algebras of other isometries in order to gain further insight into the Collatz conjecture.

It is clear that the convergence of the Collatz dynamical system $(T^k(n))_{k=1}^\infty$, $n\in\mathbb{N}^*$, depends on how the sequences of signs $\left((-1)^{T^k(n)} \right)_{n=1}^\infty$ correlate with each other. Since these sequences are periodic with period $2^k$, it suffices to consider their first parts, i.e. the sequences $p_k=\left((-1)^{T^k(\upsilon)} \right)_{\upsilon=1}^{2^k}$. In this study, we utilize the inner product of two sequences as a measure of their correlation.

These correlations can be equivalently seen in the frequency domain in the corresponding transformed sequences as follows. Recall from Section \ref{sec.periodicity} the polynomials $P_k$, $k\in \mathbb{N}$, defined as $P_k(x) = \sum_{\upsilon=1}^{2^k} (-1)^{T^k(\upsilon)} x^\upsilon$. We also need the Discrete Fourier Transform (DFT) matrices based on the roots of the cyclotomic equation $x^{2^k}+1=0$, i.e.
$$W_k= \left(\omega_{jk}^i \right)_{i=1, j=1}^{2^k, 2^k} \in \mathbb{C}^{2^k\times 2^k}.$$
The above matrices satisfy $W_k W_k^\ast = W_k^\ast W_k =2^k I_{2^k}$ and hence can be normalized by considering the matrices $\frac{1}{2^{k/2}} W_k$. Therefore, the transformed sequences, denoted by $\widetilde{p_k}$ are defined as DFT of the normalized sequence of signs $\frac{1}{2^{k/2}} p_k$, i.e.
$$\widetilde{p_k} = DFT\left(\frac{p_k}{\|p_k\|_2} \right) = \frac{1}{2^{k/2}} p_k \left(\frac{1}{2^{k/2}} W_k \right) =\frac{1}{2^k} \left( P_k (\omega_{jk}\right)_{j=1}^{2^k}.$$
One of the main results of the paper is that can be viewed collectively in terms of the isometry $S$ defined on the Hilbert space $V_\infty = V_0 \oplus V_1 \oplus V_2\oplus \ldots$ which is represented by the following matrix:
$$M= \left[ \begin{array}{ccccc}
             0 & M_0 & 0 & 0 &\ldots\\
             0 & 0 & M_1 & 0 &\ldots\\
             0 & 0 & 0 & M_2 &\ldots\\
             \vdots & \vdots & \vdots & \vdots  &\ldots
           \end{array} \right]$$
(see Section \ref{sec.ell-2-oper}). The information concerning the signs of the Collatz dynamical systems provided by the above isometry can be summarized in the following formula:
\begin{align*}
  e_1 (I-\lambda M)^{-1} = & e_1 (I+\lambda M +\lambda^2M^2 + \ldots) \\
   & = (1, \lambda \widetilde{p_1}, \lambda^2 \widetilde{p_2} , \lambda^3 \widetilde{p_3}, \ldots ),
\end{align*}
which defines the vector $\xi^T$ of all sequences of signs in the frequency domain ($\lambda$ is a fixed complex number in the unit disc, i.e. $\abs{\lambda}<1$).

The Collatz conjecture can be studied via the various autocorrelations of the vector $\xi^T$. These autocorrelations have to be matched to the autocorrelations of the dynamical system of shift type (not to be confused with the shift operator). The sequences of signs of the ideal shift are given by $p_k^o = ((-1)^k)_{\upsilon=1}^{2^k}$ and in normalized form $\frac{1}{\|p_k^o\|} p_k^o = \frac{1}{2^{k/2}} p_k^o$. The corresponding polynomials are:
$$P_k^o (x) =\sum_{\upsilon=1}^{2^k} (-1)^k x^\upsilon.$$
The corresponding transformed sequences in the frequency domain are given by:
$$\widetilde{P_k^o} = DFT \left(\frac{p_k^o}{\|p_k^o\|} \right) = \frac{1}{2^k} \left(P_k^o (\omega_{jk})\right)_{j=1}^{2^k}.$$

In the time domain, the transition map that assigns $\frac{1}{2^{k/2}} p_k^o$ to $\frac{1}{2^{\frac{k+1}{2}}} p_{k+1}^o$ is given by the matrix $-\frac{1}{\sqrt{2}} [I_{2^k}, I_{2^k}]$. Consequently, in the frequency domain the transition map between $\widetilde{p_k^o}$ and $\widetilde{p_{k+1}^o}$ is given by Fourier conjugation as follows:
$$M_k^o = \frac{1}{2^{k+1}} \left[ -W_k^\ast , -W_k^\ast \right] W_{k+1}.$$
These maps define an isometry:
$$M^o= \left[ \begin{array}{ccccc}
             0 & M_0^o & 0 & 0 &\ldots\\
             0 & 0 & M_1^o & 0 &\ldots\\
             0 & 0 & 0 & M_2^o &\ldots\\
             \vdots & \vdots & \vdots & \vdots  &\ldots
           \end{array} \right]$$
in terms of which the sequence of signs (in the frequency domain) of all levels of the ideal shift is given by:
$$\xi^o = e_1 (I-\lambda M^o)^{-1} = (1, \lambda \widetilde{p_1^o}, \lambda^2 \widetilde{p_2^o}, \lambda^3 \widetilde{p_3^o}, \ldots ).$$

\subsection*{Correlations of parity vectors in an operator theoretic setting}
The correlations of the sequences of signs of the Collatz system in various levels can be seen collectively as autocorrelations of a single infinite dimensional vector $\xi^T$. These autocorrelations have to be compared with the autocorrelations of the vector $\xi^o$ coming from the ideal shift. To this end the following two issues have to be addressed:
\begin{enumerate}
  \item An operator theoretic tool has to be derived to set up all these autocorrelations collectively.
  \item The vectors $\xi^T$, $\xi^o$ have to be broken down into a low and high part, since only the low part can assume ideal autocorrelations.
\end{enumerate}
The first requirement can be tackled with the definition of an appropriate isometry on the space $V_\infty$ which is described next. Due to the periodicity of the sequences of signs $p_k$, two consecutive sequences $\frac{1}{2^{k/2}} p_k$ and $\frac{1}{2^{\frac{k+1}{2}}} p_{k+1}$ can be correlated by sending the first one to its image via the isometry which in matrix form is given by $\frac{1}{\sqrt{2}} [I_{2^k}, -I_{2^k}]$.

In the frequency domain, the Fourier conjugate of the above map is given by:
$$C_k = \frac{1}{2^{k+1}} [ W_k^\ast, -W_k^\ast ] W_{k+1}$$
and it can be used to correlate $\widetilde{p_k}$ and $\widetilde{p_{k+1}}$ , as
$$\sca{\widetilde{p_k}C_k, \widetilde{p_{k+1}}} = \sca{\widetilde{p_k}, \widetilde{p_{k+1}} C_k^\ast}.$$
Hence, one can correlate $\widetilde{p_k}$ and $\widetilde{p_{k+\lambda}}$ via the composition
$$C_k \circ C_{k+1} \circ \ldots \circ C_{k+l-1}.$$
Consequently, all autocorrelations of $\xi$ vectors can be collectively viewed via the isometry
$$C=  \left[ \begin{array}{ccccc}
             0 & C_0 & 0 & 0 &\ldots\\
             0 & 0 & C_1 & 0 &\ldots\\
             0 & 0 & 0 & C_2 &\ldots\\
             \vdots & \vdots & \vdots & \vdots  &\ldots
           \end{array} \right]$$
by invoking expressions of the form $\frac{1}{\|\xi\|^2} \sca{\xi C^{\ast \nu}, \xi C^{\ast \mu}}$.

As far as the decomposition of $\xi$ in low and high part, this is achieved as follows. We fix an integer $r$, with $1<r<2^k$ and we decompose $\widetilde{p_k}$ (and similarly $\widetilde{p_{k+1}}$) into two vectors:
$$\widetilde{p_k} = \Delta_L^r \widetilde{p_k} +\Delta_H^r \widetilde{p_k} = \frac{1}{2^{k/2}} DFT\left( id_{\le r} (p_k) \right) + \frac{1}{2^{k/2}} DFT\left( id_{> r} (p_k) \right),$$
where $id_{\le r} \left((-1)^{T^k(\upsilon)}\right)_{\upsilon=1}^{2^k} = \left((-1)^{T^k(1)}, (-1)^{T^k(2)} , \ldots , (-1)^{T^k(r)}, 0, \ldots ,0 \right)$ and $id_{>r}$ is defined in the obvious way. Therefore, for a sequence $r=(r_i)_{i=1}^\infty$, with $1< r_i <2^i$ for any $i\in\mathbb{N}$, we can define the vectors
\begin{align*}
  \xi^T_r  = &  \left(1, \lambda \Delta_L^{r_1} \widetilde{p_1} , \lambda^2  \Delta_L^{r_2} \widetilde{p_2}, \lambda^3 \Delta_L^{r_3} \widetilde{p_3} ,\ldots \right)\\
  \xi^o_r  = &  \left(1, \lambda \Delta_L^{r_1} \widetilde{p_1^o} , \lambda^2  \Delta_L^{r_2} \widetilde{p_2^o}, \lambda^3 \Delta_L^{r_3} \widetilde{p_3^o} ,\ldots \right).
\end{align*}

\subsection*{Collatz powers and correlations}
The existence of a single attractor to which all Collatz trajectories converge implies that $T^{f(n)}(n) \in \{1,2\}$ for some function $f\colon \mathbb{N}^* \to \mathbb{N}^*$. Conversely, given $f\colon \mathbb{N}^* \to \mathbb{N}^*$ we are interested in assess the convergence properties of $f$ acting by exponentiation on $T$. We may assume that $f$ is increasing and that $f(n) \ge \log_2 n$. Denote by $Inc$ the set of all these functions, i.e.
$$Inc = \{ f\colon \mathbb{N}^* \to \mathbb{N}^* \mid f \text{ is increasing and } f(n) \ge \log_2 n \,\, \forall n \in\mathbb{N}^*\}.$$
Any function belonging to the above set is piecewise constant in a wider sense and therefore it is determined by two strictly increasing sequences:
$$x(f) = (a_1, a_2, a_3, \ldots) \quad \text{and} \quad y(f) = (b_1,b_2,b_3, \ldots ),$$
as follows:
$$f(n) = y(f)_i =b_i \quad \text{for any } i \text{ with } a_i \le n <a_{i+1}.$$

Any function $f\in Inc$ defines a new increasing function $f^\ast \colon \mathbb{N}^* \to \mathbb{N}^*$ such that $x(f^\ast) =y(f) $ and $y(f^\ast)_i = x(f)_{i+1}-1$ for any $i\ge 0$. In this setting, we call $f^\ast$ the conjugate of $f$. It is not hard to observe that for any positive integer $n$, we have:
$$ f(f^\ast(n)) \le n \quad \text{ and } \quad f^\ast(f(n)) \ge n.$$
Indeed, for the first inequality, assume that $b_i \le n <b_{i+1}$ for some $i$. Then $f^\ast(n) = a_{i+1}-1 <a_{i+1}$, which implies that $f(f^\ast(n)) \le b_i \le n$.

As far as the second inequality is concerned, we assume that $a_i \le n<a_{i+1}$ for some $i$. Then, $f(n)=b_i$ which implies that $f^\ast(f(n))=f^\ast(b_i)=a_{i+1}-1\ge n$.

The functions $f$ and $f^\ast$ have the following duality property concerning the solvability of the Collatz conjecture.

\begin{theorem}
  Assume that $f\in Inc$. The following are equivalent.
  \begin{enumerate}
    \item $T^{f(n)}(n) \in \{1,2\}$ for every $n\in \mathbb{N}^*$.
    \item For every $n_1 \le n_2$ in $\mathbb{N}^*$, we have
    $$ \sca{id_{\le f^\ast(n_1)} \left(p_{n_1}\right), id_{\le f^\ast(n_2)} \left(p_{n_2}\right)} = (-1)^{n_2-n_1} f^\ast (n_1).$$
  \end{enumerate}
\end{theorem}

\begin{proof}
  In order to prove the implication $(1)\Rightarrow(2)$, we firstly observe that for any $n,k\in\mathbb{N}^*$ with $k\le f^\ast(n)$, we have $T^n(k) \in \{1,2\}$. Indeed, since $f$ is increasing, we obtain that $f(k) \le f(f^\ast(n)) \le n$. However, $T^{f(k)}(k)\in\{1,2\}$ and, thus, $T^n(k) \in\{1,2\}$.

  Secondly, if $T^n(k) \in \{1,2\}$ for some $n\in\mathbb{N}^*$, then the sequence $(T^{n+i}(k))_{i=0}^\infty$ is 2-periodic taking successively the values $1,2$. This easily implies that $T^{n+i}(k) - T^n(k) = i \mod 2$ for any $i\ge 0$. Therefore,
  $$(-1)^{T^{n+i}(k) + T^n(k)} =(-1)^i.$$

  By the above remarks, we deduce that
  \begin{align*}
    \sca{id_{\le f^\ast(n_1)} \left(p_{n_1}\right), id_{\le f^\ast(n_2)} \left(p_{n_2}\right)} & = \sum_{1\le k \le f^\ast(n_1)} (-1)^{T^{n_2}(k)+T^{n_1}(k)} \\
    = & \sum_{1\le k \le f^\ast(n_1)} (-1)^{T^{n_1+n_2-n_1}(k)+T^{n_1}(k)}\\
    = & \sum_{1\le k \le f^\ast(n_1)} (-1)^{n_2-n_1}=  (-1)^{n_2-n_1} f^\ast(n_1)
  \end{align*}

  For the reverse implication, we fix $n\in\mathbb{N}^*$ and we set $n_1=f(n)$. By Corollary \ref{cor.periodicity}, it suffices to prove that the sequence $\left((-1)^{T^{n_1+i}(n)}\right)_{i=0}^\infty$ is periodic with period equal to $2$. For an $i\in\mathbb{N}$, we set $n_2=n_1+i$ and by our hypothesis we obtain:
  $$\sca{id_{\le f^\ast(n_1)} \left(p_{n_1}\right), id_{\le f^\ast(n_2)} \left(p_{n_2}\right)} = (-1)^{n_2-n_1} f^\ast (n_1),$$
  i.e.
  $$\sum_{k=1}^{f^\ast(n_1)} (-1)^{T^{n_2}(k)+T^{n_1}(k)} = (-1)^i f^\ast(n_1).$$
  The last equation implies that
  $$(-1)^{T^{n_2}(k)+T^{n_1}(k)} = (-1)^i \quad \forall k=1,2,\ldots, f^\ast(n_1).$$
  Since $n_1=f(n)$, it follows that $f^\ast(n_1) = f^\ast(f(n)) \ge n$. Consequently, the above equation holds true for $k=n\le f^\ast(n_1)$, and we obtain that
  $$(-1)^{T^{n_1+i}(n)}\cdot (-1)^{T^{n_1}(n)} = (-1)^i.$$
  Therefore, the sequence $\left((-1)^{T^{n_1+i}(n)}\right)_{i=1}^\infty$ is periodic with period $2$ and the proof is complete.
\end{proof}

\begin{corollary}
    Assume that $f\in Inc$. The following are equivalent.
  \begin{enumerate}
    \item $T^{f(n)}(n) \in \{1,2\}$ for every $n\in \mathbb{N}^*$.
    \item For every $n_1<n_2$ in $\mathbb{N}^*$, we have
    $$ \sca{\Delta_L^{f^\ast(n_1)} \widetilde{p_{n_1}} C_{n_1} C_{n_1+1} \ldots C_{n_2-1} , \Delta_L^{f^\ast(n_2)}\widetilde{p_{n_2}}} = \frac{(-1)^{n_2-n_1} f^\ast(n_1)}{2^{\frac{n_1+n_2}{2}}}.$$
  \end{enumerate}
\end{corollary}

\subsection*{Solvability of Collatz in terms of a net of $C^\ast$-algebra functionals. Final results.}
The evolution and convergence of Collatz orbits depend on the behavior of powers of $T$, i.e. $T^k(n)$. These powers, in the setting of the present paper, can be described collectively in terms of the operator $M$ which is an isometry. All the powers of $M$ as well as their linear combinations together with inverse powers, which are depicted by power of $M^*$, are contained in an algebra structure $A(S)$, the $C^*$-algebra of an isometry. This algebra contains elements of the form $a(S, S^*)= \sum a_{ij} (S^*)^i S^j$. We will characterize the solvability of Collatz in terms of this structure. In order to achieve this goal, we need to define appropriate linear functionals:
$$\varphi_f \colon A(S) \to \mathbb{C}$$
which contain information of the Collatz dynamical system and the convergence ability of powers $T^f$ for $f\in Inc$.

To this end, for every $f\in Inc$ we define the vectors:
\begin{align*}
  \xi_f^T = & \left(1, \lambda \Delta_L^{f^\ast(1)} (\widetilde{p}_1), \lambda^2 \Delta_L^{f^\ast(2)} (\widetilde{p}_2), \lambda^3 \Delta_L^{f^\ast(3)} (\widetilde{p}_3), \ldots \right) \\
  \xi_f^o = & \left(1, \lambda \Delta_L^{f^\ast(1)} (\widetilde{p}_1^o, \lambda^2 \Delta_L^{f^\ast(2)} (\widetilde{p}_2^o), \lambda^3 \Delta_L^{f^\ast(3)} (\widetilde{p}_3^o), \ldots \right)
\end{align*}
and the functional on $A(S)$
$$\varphi_f \left(a(S, S^*) \right) = \frac{1}{\|\xi_f^T\|^2} \sca{\xi_f^T, \xi_f^T a(C,C^*)} - \frac{1}{\|\xi_f^o\|^2} \sca{\xi_f^o, \xi_f^o a(C,C^*)}. $$
Then $\varphi_f$ is a functional that compares the correlations of the sequences of signs of $T$ with those of the ideal shift. In the case where these two coincide for some $f\in Inc$, then $T^{f(n)}(n) \in \{1,2\}$, for every $n$, and thus the Collatz conjecture is true. Consequently, we obtain the next result.

\begin{theorem}
  The Collatz conjecture is true if and only if the set $\{f\in Inc \mid \varphi_f=0\}$ is non empty.
\end{theorem}

\begin{corollary}
  If the set $\{f\in Inc \mid \varphi_f=0\}$ is non empty, then this set has a minimum element $f_0\in Inc$, which is
  $$f_0(n) = \max\{s(i) \mid i\le n\}.$$
\end{corollary}



\bibliographystyle{amsplain}

\end{document}